\documentclass[a4paper,11pt]{amsart}
\usepackage[latin1]{inputenc}
\usepackage[hyphens]{url}
\usepackage{amsmath,amsrefs,amssymb,amsthm}
\usepackage{graphicx}
\usepackage[bookmarks=true,pdfstartview={FitH}]{hyperref}
\usepackage{enumerate,enumitem}
\usepackage{mathrsfs}
\usepackage{epstopdf}
\usepackage{fullpage}
\usepackage[small]{caption}
\newtheorem{defi}{Definition}[section]
\newtheorem{lem}[defi]{Lemma}
\newtheorem{thm}[defi]{Theorem}
\newtheorem{cor}[defi]{Corollary}
\newtheorem{prop}[defi]{Proposition}
\theoremstyle{remark}
\newtheorem{rem}[defi]{Remark}
\numberwithin{equation}{section}

\def\Cc{\ensuremath{\mathscr{C}}}
\def\C{\ensuremath{\mathcal{C}}}
\def\dis{\ensuremath{{\rm dis}\;}}
\def\Dd{\ensuremath{\mathfrak{D}}}
\def\dP{\ensuremath{d_{\rm P}}}
\def\dGP{\ensuremath{d_{\rm GP}}}
\def\dMGP{\ensuremath{d_{\rm mGP}}}
\def\dto{\ensuremath{\stackrel{d}{\to}}}
\def\dust{\ensuremath{\mathcal{M}_{\rm{dust}}}}
\def\E{\ensuremath{\mathbb{E}}}
\def\ep{\varepsilon}
\def\F{\ensuremath{\mathcal{F}}}
\def\p{\ensuremath{\mathcal{P}}}
\def\P{\ensuremath{\mathbb{P}}}

\def\Mm{\ensuremath{\mathcal{M}}}

\def\N{\ensuremath{\mathbb{N}}}
\def\nd{\ensuremath{\mathcal{M}_{\rm{nd}}}}
\def\R{\ensuremath{\mathbb{R}}}
\def\fR{\ensuremath{\mathfrak{R}}}
\def\S{\ensuremath{\mathcal{S}}}

\def\sk{\ensuremath{[k]}}
\def\sN{\ensuremath{[N]}}
\def\sn{\ensuremath{[n]}}

\def\supp{\ensuremath{{\rm supp}\;}}
\def\U{\ensuremath{\mathbb{U}}}
\def\Uu{\ensuremath{\mathfrak{U}}}
\def\UUerg{\ensuremath{\mathcal{U}^{\rm erg}}}

\def\UU{\ensuremath{\mathcal{U}}}
\def\wto{\ensuremath{\stackrel{w}{\to}}}
\def\utwo{\ensuremath{\underline{\underline{2}}}}
\def\uone{\ensuremath{\underline{1}}}

\newcommand{\reff}[1]{(\ref{#1})}
\newcommand{\expp}[1]{\mathop {e^{ #1}}}
\newcommand{\ew}[1]{\mathop {[\![{ #1}]\!]}}

\newcommand{\I}[1]{\mathop {\mathbf{1}{\left\{ #1\right\}}}}
\newcommand{\fl}[1]{\mathop {\lfloor{ #1}\rfloor}}
\newcommand{\hparagraph}[1]{\par\noindent\textbf{#1}\mbox{}\\}

\begin{document}
\title{Invariance principles for tree-valued Cannings chains}
\label{Conv}
\author{Stephan Gufler}
\address{Goethe-Universität, Institut für Mathematik, Postfach 11 19 32, Fach 187, 60054 Frankfurt am Main, Germany}
\email{gufler@math.uni-frankfurt.de}

\begin{abstract}
We consider sequences of tree-valued Markov chains that describe evolving genealogies in Cannings models, and we show their convergence in distribution to tree-valued Fleming-Viot processes.
Under the conditions of Möhle and Sagitov, this convergence holds for all tree-valued Fleming-Viot processes under consideration in the dust-free case, and for the Fleming-Viot processes with values in the space of distance matrix distributions in the case with dust.
Convergence to Fleming-Viot processes with values in the space of marked metric measure spaces in the case with dust is ensured by an additional assumption on the probability that a randomly sampled individual belongs to a non-singleton family.
\end{abstract}
\keywords{Invariance principle, Cannings model, tree-valued Fleming-Viot process, (marked) metric measure space, (marked) Gromov-Prohorov metric, external branches.}
\subjclass[2010]{Primary 60F17, Secondary 60J25, 60K35, 92D10}

\maketitle
{\footnotesize \tableofcontents }

\section{Introduction}
In population genetics, Cannings models \cites{Cannings1,Cannings2} are classical evolutionary models with constant population size $N$ and non-overlapping generations. Reproduction events occur independently between the generations such that the individuals in generation $k+1$ are subdivided into families according to an exchangeable random partition and each family draws its ancestor in generation $k$ independently without replacement. The Wright-Fisher model is a classical example for a Cannings model. In the Wright-Fisher model, each individual in generation $k+1$ draws its ancestor in generation $k$ independently with replacement.
We refer the reader also to e.\,g.\ \cite{Eth11} for these models.

Coalescents with simultaneous multiple mergers are robust infinite population size limits of partition-valued processes that describe the genealogies in Cannings models at fixed times. Möhle and Sagitov \cite{MS01} give a criterion for this convergence. Sagitov \cite{Sag03} gives an equivalent criterion in terms of the measure $\Xi$ of Schweinsberg \cite{Schw00}. A first robustness result for the Kingman coalescent is shown in Kingman \cite{K82}, see also e.\,g.\  \cite{Bi05}*{Theorem 2}.
 
In the present article, we consider evolving genealogies, and we describe the genealogical tree of the individuals in Cannings models at each time in various ways.
First, we consider the metric measure space that consists of the set of individuals at that time, their mutual genealogical distances, and the uniform probability measure. Second, we consider the distance matrix distribution of the aforementioned metric measure space, i.\,e.\ the distribution of the infinite matrix of the genealogical distances between iid samples.
Third, we decompose the genealogical tree into the external branches and the remaining subtree. We then consider the semi-metric space that consists of the starting vertices of the external branches and their mutual distances. On the product space of this semi-metric space with $\R_+$, we define a probability measure such that for each external branch, mass $1/N$ is added to the pair that consists of its starting vertex and its length. We then obtain a marked metric measure space.
Fourth, we also decompose the genealogical tree and describe it by a marked metric measure space, pruning not the whole external branch, but only the part from each leaf to the most recent reproducing individual on the ancestral lineage.

We endow the space of distance matrix distributions with the Prohorov metric, and the space of (marked) metric measure spaces with the (marked) Gromov-Prohorov metric \cites{GPW09,DGP11}.
We consider Markov chains whose states describe the evolving genealogy in Cannings models in one of the four aforementioned ways. Our invariance principles show that sequences of such Markov chains converge in distribution, under a rescaling of time and genealogical distances, in the space of càdlàg paths in the respective state space, endowed with the Skorohod topology.
The limit processes are the Fleming-Viot processes from \cite{Sampl}*{Section 8}.
Under the condition of Sagitov \cite{Sag03} and the assumption that the initial states converge, we show the convergence of the prelimiting chains with values in the space of distance matrices.
For the convergence of the prelimiting chains with values in the space of metric measure spaces, we have to assume in addition that the limiting genealogy is dust-free as the tree-valued Fleming-Viot process with values in the space of metric measure spaces exists only in the dust-free case.
Dust-freeness can be characterized by the property that a randomly drawn external branch has a.\,s.\ length zero
(cf.\ Propositions 6.6 and 7.4 in \cite{Sampl}).

In the dust-free case, the sequences of prelimiting chains for the third and the fourth description of the genealogy converge under the condition of Sagitov \cite{Sag03} and an appropriate condition on the initial state. We show their convergence in the case with dust under the additional assumption \reff{Conv:eq:singl-dust} on the probability that a randomly sampled individual from a fixed generation belongs to a non-singleton family. An additional assumption is needed here as the convergence in the marked Gromov-Prohorov metric (other than the weak convergence of the distance matrix distributions) implies weak convergence of the empirical distribution of the external branch lengths or the distances to the most recent reproducing individual, respectively.

In Section~\ref{Conv:sec:dist-dec}, we recall the decomposition of the genealogical trees at the external branches. In Section~\ref{Conv:sec:prelim:mm}, we recall some notions on metric measure spaces and marked metric measure spaces. We state our convergence results in Section~\ref{Conv:sec:results}. We recall tree-valued $\Xi$-Fleming-Viot processes in Section~\ref{Conv:sec:TVFV}. In Section~\ref{Conv:sec:counterex}, we give an example in which assumption \reff{Conv:eq:singl-dust} is not satisfied and the chains with values in the space of marked metric measure spaces do not converge. The proofs of the invariance principles are given in the further sections.

The tree-valued Fleming-Viot process with binary reproduction events is introduced in Greven, Pfaffelhuber, and Winter \cite{GPW13} as the solution of a well-posed martingale problem that is the limit in distribution of tree-valued processes read off from Moran models. In \cite{GPW13}*{Remark 2.21}, it is conjectured that a tree-valued Fleming-Viot process is the robust limit of tree-valued processes read off from Cannings models.
In \cite{Sampl}, tree-valued Fleming-Viot processes are studied in the setting with simultaneous multiple reproduction events, the case with dust is included by the decomposition of the genealogical trees into the external branches and the remaining subtree. These decomposed genealogical trees are described by marked metric measure spaces and their distance matrix distributions. Path regularity of tree-valued Fleming-Viot processes follows from the pathwise construction in \cite{Pathw}.

In Section~\ref{Conv:sec:conv-gen}, we prove the invariance principles for the Markov chains associated with the first, second, and in the case with dust also for the fourth of the above descriptions of the genealogy. 
Here we can apply a general convergence result from Ethier and Kurtz \cite{EK86}*{Chapter 4.8} as the transition kernels of the prelimiting chains converge on a core (in the sense of \cite{EK86}*{Chapter 1.3}) to the generators of the tree-valued Fleming-Viot processes. In \cite{Sampl}, it is shown that the domains of the martingale problems for the tree-valued Fleming-Viot processes are cores, and that the semigroups on these cores are strongly continuous. We use existence and path regularity of the limit processes, and we do not need to show relative compactness of the prelimiting processes here.

We prove the invariance principles for the processes with values in the space of marked metric measure spaces in Sections~\ref{Conv:sec:mmm-nd} and~\ref{Conv:sec:mmm-dust} by comparison with processes whose convergence is proved in Section~\ref{Conv:sec:conv-gen}. In the case without dust, we can compare these process in the supremum metric. In the case with dust, we compare only the finite dimensional distributions whence we also have to check then the relative compactness of the prelimiting processes.

We use also exchangeable random partitions, Kingman's correspondence and its continuity properties, for which we refer to Pitman \cite{Pitman06}*{Chapter 2} and Bertoin \cite{Bertoin}*{Chapter 2}.
We mention that Stournaras \cite{Stournaras12} uses \cite{EK86}*{Corollary 4.8.17} to show, by verifying the compact containment condition via \cite{GPW13}*{Proposition 2.22} with some effort, convergence of tree-valued Wright-Fisher models to the tree-valued Fleming-Viot process from \cite{GPW13}.

\section{Preliminaries}
It is well-known that ultrametric spaces can be viewed as leaf-labeled real trees (cf.\ e.\,g.\ \cite{Sampl}*{Remark 1.1}).
In Subsection \ref{Conv:sec:dist-dec}, we recall a decomposition of semi-ultrametrics that corresponds to the decomposition of the associated trees at the external branches.
In Subsection \ref{Conv:sec:prelim:mm}, we recall isomorphy classes of metric measure spaces and marked metric measure spaces which we can be interpreted as unlabeled genealogical trees. In the finite case, the (marked) metric measure space associated with a (decomposed) ultrametric can be viewed as the equivalence class under permutations of the labels of the leaves. 

\subsection{Distance matrices}
\label{Conv:sec:dist-dec}
We denote the set of the positive integers by $\N$, the set of the non-negative integers by $\N_0$, and for $N\in\N$, we write $\sN=\{1,\ldots,N\}$. Let $\Dd$ denote the space of semimetrics on $\N$ and $\Uu\subset\Dd$\label{Conv:not:Uu} the set of semi-ultrametrics on $\N$. We do not distinguish between a semi-metric $\rho\in\Dd$ and the distance matrix $(\rho(i,j))_{i,j\in\N}$, and we view $\Uu$ and $\Dd$ as subspaces of the space $\R^{\N^2}$ which we endow with a complete and separable metric that induces the product topology, where $\R$ is endowed with the Euclidean topology. Analogously, for $N\in\N$, we denote by $\Dd_N$ the space of semimetrics on $[N]$ and by $\Uu_N\subset\Dd_N$\label{Conv:not:UuN} the space of semi-ultrametrics on $[N]$. Again we do not distinguish between semi-metrics and distance matrices and we view $\Dd_N$ as a subspace of the finite-dimensional space $\R^{N^2}$ which we endow with a norm and the (induced) Euclidean topology.

We now decompose semi-ultrametrics in $\Uu_N$ as in \cite{Sampl}*{Section 2}. The continuous map\label{Conv:not:alphaN}
\[\alpha:\R_+^{N^2}\times\R_+^{N}\to\R_+^{N^2},\quad
(r,v)\mapsto((v(i)+r(i,j)+v(j))\I{i\neq j})_{i,j\in[N]},\]
is used to retrieve ultrametric distance matrices from the elements of the space\label{Conv:not:hatUuN}
\[\hat\Uu_N=\{(r,v)\in\Dd_N\times\R_+^{N}:\alpha(r,v)\in\Uu_N\}\]
of decomposed semi-ultrametrics which we also call marked distance matrices.
Conversely, for $N\geq 2$, we decompose ultrametric distance matrices using the maps\label{Conv:not:UpsilonN}
\[\Upsilon:\Uu_N\to\R_+^N,\quad\rho\mapsto(\tfrac{1}{2}\min_{j\in[N]\setminus\{i\}}\rho(i,j))_{i\in[N]}\]
and\label{Conv:not:beta}
$\beta:\Uu_N\to\hat\Uu_N$, $\rho\mapsto(r,v)$,
where $v=\Upsilon(\rho)$ and
$r(i,j)=(\rho(i,j)-v(i)-v(j))\I{i\neq j}$
for $i,j\in[N]$. For $\rho\in\Uu_1$, we set $\Upsilon(\rho)=0$ and $\beta(\rho)=(\rho,0)$.

As in \cite{Sampl}*{Remark 2.2}, the quantity $v(i)$ is the length of the external branch that ends in leaf $i$ of the coalescent tree associated with $\rho$, and $r(i,j)$ is the distance between the starting vertices of the external branches that end in leaves $i$ and $j$, respectively. Here an external branch is defined to consist only of the leaf $i$ if there exists $j\in[N]\setminus\{i\}$ with $\rho(i,j)=0$. (In fact, the finite setting discussed in this subsection can be seen as a special case of Section 2 in \cite{Sampl} as any semi-ultrametric $\rho$ on $[N]$ can be extended to $\N$ by setting e.\,g.\ $\rho(1,k)=0$ for $k>N$.)

We also use the decomposition of semi-ultrametrics in $\Uu$ from \cite{Sampl}*{Section 2}. Here we have the continuous map\label{Conv:not:alpha}
\[\alpha:\R_+^{\N^2}\times\R_+^{\N}\to\R_+^{\N^2},
\quad(r,v)\mapsto((v(i)+r(i,j)+v(j))\I{i\neq j})_{i,j\in\N},\]
the space $\hat\Uu=\{(r,v)\in\Dd\times\R_+^\N:\alpha(r,v)\in\Uu\}\subset\R^{\N^2}\times\R^\N$\label{Conv:not:hatUu} of marked distance matrices (or decomposed semi-ultrametrics), and the map\label{Conv:not:Upsilon}
\[\Upsilon:\Uu\to\R_+^{\N},\quad\rho\mapsto(\tfrac{1}{2}\inf_{j\in\N\setminus\{i\}}\rho(i,j))_{i\in\N}\]
which maps a semi-ultrametric $\rho$ to the sequence of the external branch lengths of the associated tree.

\subsection{Metric measure spaces and marked metric measure spaces}
\label{Conv:sec:prelim:mm}
For the theory of metric measure spaces and marked metric measure spaces, we refer to \cites{Gromov,Vershik02,GPW09,DGP11,Lohr13}.

A metric measure space is a triple $(X,\rho,\mu)$ that consists of a complete and separable metric space $(X,\rho)$ and a probability measure $\mu$ on the Borel sigma algebra on $X$. Two metric measure spaces $(X,\rho,\mu)$ and $(X',\rho',\mu')$ are defined to be isomorphic if there exists a measure-preserving isometry between the supports $\supp\mu$ and $\supp\mu'$.
The distance matrix distribution\label{Conv:not:nuchimm} $\nu^{(X,\rho,\mu)}$ of a metric measure space $(X,\rho,\mu)$ is defined as the distribution of the random matrix $(\rho(x(i),x(j)))_{i,j\in\N}$, where $(x(i),i\in\N)$ is a $\mu$-iid sequence in $X$.
The Gromov reconstruction theorem (\cite{Vershik02}*{Theorem 4} and \cite{Gromov}*{Section $3\tfrac{1}{2}$}) states that metric measure spaces are isomorphic if and only if they have the same distance matrix distribution.
Hence, we can define the isomorphy class $\ew{X,\rho,\mu}$ of a metric measure space $(X,\rho,\mu)$ such that\label{Conv:not:U}
\[\U=\{\ew{X',\rho',\mu'}:(X',\rho',\mu')\text{ ultrametric measure space}\}\]
is a set. For $\chi\in\U$, we denote the associated distance matrix distribution by $\nu^\chi$. A sequence $(\chi_k,k\in\N)$ in $\U$ converges to $\chi\in\U$ in the Gromov-weak topology if and only if the distance matrix distributions $\nu^{\chi_k}$ converge weakly to $\nu^\chi$. The Gromov-Prohorov metric\label{Conv:not:dGP} $\dGP$ induces the Gromov-weak topology and is complete and separable, as shown in \cites{GPW09}.
As in \cite{GPW13}, the elements of $\U$ can be considered as trees.

For $N\in\N$, we also work with the closed subspace\label{Conv:not:UN}
\[\U_N=\{\ew{X,\rho,\mu}:(X,\rho,\mu)\text{ ultrametric measure space such that $N\mu$ is integer-valued}\}\]
of $\U$ which can be interpreted as the space of semi-ultrametric spaces that contain $N$ elements and that are endowed with the uniform probability measure.
When we identify points $x,y\in X$ with $\rho(x,y)=0$ in a semi-metric space $(X,\rho)$ to obtain a metric space, we refer by $x,y$ also to the corresponding element of the metric space, in slight abuse of notation. 
We define the isomorphy class of a semi-metric measure space as the isomorphy class of the metric measure space obtained by identifying the elements with distance zero.
We define the function\label{Conv:not:psi}
\[\psi_N:\Uu_N\to\U_N,\quad\rho\mapsto\ew{\sN,\rho,N^{-1}\sum_{i=1}^N\delta_i}\]
which maps a semi-ultrametric to the isomorphy class of the associated semi-metric measure space with the uniform measure. It is clear that the map $\psi_N$ is continuous, for formal proofs, cf.\ \cite{Sampl}*{Remark 11.1} or \cite{Pathw}*{Lemma 4.5}.

For each $\chi\in\U_N$, there exists $\rho\in\Uu_N$ with $\psi_N(\rho)=\chi$, and the $N$-distance matrix distribution\label{Conv:not:nuNchimm} $\nu^{N,\chi}$ is defined as the distribution of the random matrix $(\rho(x(i),x(j)))_{i,j\in\sN}$, where $x(1),\ldots,x(N)$ are sampled from $\sN$ according to the uniform measure without replacement.
(That is, $(x(i),i\in[N])$ is a uniform permutation of $[N]$.)
For every $\Uu_N$-valued random variable $\rho'$ with distribution $\nu^{N,\chi}$, it holds $\psi_N(\rho')=\chi$ a.\,s. Hence, $\chi$ is uniquely determined by $\nu^{N,\chi}$, as in \cite{GPW13}.

A ($\R_+$-)marked metric measure space is a triple $(X,r,m)$ that consists of a complete and separable metric space $(X,r)$ and a probability measure $m$ on the Borel sigma algebra on the space $X\times\R_+$ which is endowed with the product metric $d((x,v),(x',v'))=r(x,x')\vee|v-v'|$.
Two marked metric measure spaces $(X,r,m)$ and $(X',r',m')$ are defined to be isomorphic if there exists an isometry $\varphi$ between the supports $\supp m(\cdot\times\R_+)$ and $\supp m'(\cdot\times\R_+)$ such that the isometry
\[\hat\varphi:\supp m\to\supp m',\quad
(x,v)\mapsto(\varphi(x),v)\]
satisfies $\hat\varphi(m)=m'$.
The marked distance matrix distribution $\nu^{(X,r,m)}$ of a marked metric measure space $(X,r,m)$ is defined as the distribution of $((r(x(i),x(j)))_{i,j\in\N},(v(i))_{i\in\N})$, where $((x(i),v(i)),i\in\N)$ is an $m$-iid sequence in $X\times\R_+$.
The Gromov reconstruction theorem for marked metric measure spaces (see \cite{Sampl}*{Proposition 3.12}, \cite{DGP11}*{Theorem 1}) states that marked metric measure spaces are isomorphic if and only if they have the same marked distance matrix distribution.
We can now define the isomorphy class $\ew{X,r,m}$ of a marked metric measure space $(X,r,m)$ such that\label{Conv:not:hatU}
\[\hat\U=\{\ew{X',r',m'}:\ (X',r',m') \text{ marked metric measure space with }\nu^{(X',r',m')})(\hat\Uu)=1\}\]
is a set. As in \cite{Sampl}, $\hat\U$ is the set of isomorphy classes of marked metric measure spaces that yield ultrametric spaces when marks in the support of the measure are added to the distances of the metric space, and the elements of $\hat\U$ can be viewed as trees.

We denote the marked distance matrix distribution associated with any $\chi\in\hat\U$ by\label{Conv:not:nuchimmm} $\nu^\chi$. Using the continuous map $\alpha$ from Subsection \ref{Conv:sec:dist-dec}, we associate with $\chi$ the probability distribution $\alpha(\nu^\chi)$ on $\Uu$ which is called in \cite{Sampl} the distance matrix distribution of $\chi$. (We denote by $f(\mu)$ the image measure of a measure $\mu$ under a map $f$).
A sequence $({\chi}_k,k\in\N)$ converges to $\chi$ in $\hat\U$ in the marked Gromov-weak topology if and only if the marked distance matrix distributions $\nu^{\chi_k}$ converge weakly to $\nu^{\chi}$. The marked Gromov-weak topology is metrized by the marked Gromov-Prohorov metric\label{Conv:not:dMGP} $\dMGP$ which is complete and separable, see \cite{DGP11}.

By \cite{Sampl}*{Proposition 3.3} each element of $\hat\U$ is uniquely characterized in $\hat\U$ by its distance matrix distribution. Hence, the set of distance matrix distributions of marked metric measure spaces, denoted (as in \cite{Sampl}) by\label{Conv:not:UUerg}
\[\UUerg:=\{\alpha(\nu^\chi):\chi\in\hat\U\},\]
is in one-to-one correspondence with $\hat\U$, and its elements can likewise be viewed as trees. We endow $\UUerg$ with the Prohorov metric $\dP$. Then $\UUerg$ is separable and by \cite{Sampl}*{Corollary 3.25} complete.

We say a a marked metric measure $(X,r,m)$ space supports only the zero mark if $m=\mu\otimes\delta_0$ for some probability measure $\mu$ on $X$. Clearly, this property depends only on the isomorphy class of the marked metric measure space. Also note that the distance matrix distribution of a marked metric measure space $(X,r,\mu\otimes\delta_0)$ that supports only the zero mark equals the distance matrix distribution of the metric measure space $(X,r,\mu)$.

For $N\in\N$, we also define the closed subspace\label{Conv:not:hatUN}
\[\hat\U_N=\{\ew{X,r,m}\in\hat\U:(X,r,m)\text{ marked metric measure space}, Nm \text{ integer-valued}\}\]
of $\hat\U$ that stands for finite marked metric measure spaces. To obtain marked metric measure spaces from marked distance matrices, we use the map\label{Conv:not:hatpsi}
\[\hat\psi_N:\hat\Uu_N\to\hat\U_N, (r,v)\mapsto\ew{\sN,r,N^{-1}\sum_{i=1}^N\delta_{(i,v(i))}},\]
where we understand the isomorphy class of a marked semi-metric measure space as the isomorphy class of the marked metric measure space obtained by identifying elements of the metric space with distance zero. Clearly, the map $\hat\psi_N$ is continuous, a formal proof is given in \cite{Sampl}*{Remark 11.1}.

As in \cite{DGP12}, the $N$-marked distance matrix distribution\label{Conv:not:nuNchimmm} of $\chi\in\hat\U_N$ is defined as the distribution of $((r(x(i),x(j)))_{i,j\in\sN},(v(x(i))_{i\in\sN})$, where $(r,v)$ is any element of $\hat\Uu_N$ with $\hat\psi_N(r,v)=\chi$, and $x(1),\ldots,x(N)$ is sampled uniformly from $\sN$ without replacement. Clearly, $\hat\psi_N(r',v')=\chi$ a.\,s.\ for any random variable $(r',v')$ that has the marked distance matrix distribution of $\chi$.

\section{Invariance principles}
\label{Conv:sec:results}
In Subsection~\ref{Conv:sec:Cannings}, we define for each $N\in\N$ the Cannings population model of population size $N$. From this construction, we read off tree-valued processes in Subsections~\ref{Conv:sec:inv-mm} --~\ref{Conv:sec:mod}.

\subsection{The Cannings model and the process of the genealogical distances}
\label{Conv:sec:Cannings}
The population model has discrete generations, enumerated by $\N_0$, and $N$ individuals in each generation, labeled by $1,\ldots,N$.
The dynamics is characterized by a probability measure $\Xi^N$ on the subspace
\[\Delta^N=\{x\in\Delta:\left| x\right|_1=1,Nx(i)\in\N_0\text{ for all }i\in\N\}\]
of the simplex
\begin{equation*}
\Delta=\{x=(x(1),x(2),\ldots): x(1)\geq x(2)\geq\ldots\geq 0, \left|x\right|_1\leq 1\},
\end{equation*}
where we write $\left|x\right|_p=(\sum_{i\in\N}\left|x(i)\right|^p)^{1/p}$.
 
First, we sample a $\Xi^N$-iid sequence\label{Conv:not:xNk} $(x^N_k,k\in\N)$ in $\Delta^N$.
Then, conditionally given $(x^N_k,k\in\N)$, let\label{Conv:not:piNk} $(\pi^N_k,k\in\N)$ be a sequence of independent random partitions of $[N]$ such that for each $k\in\N$, the partition $\pi^N_k$ is uniformly distributed on the set of partitions of $[N]$ whose block sizes are given by (any reordering of) $Nx^N_k$.
In each generation $k\in\N$ of the population model, we partition the individuals into families, saying that individuals are in the same family if their labels are in the same block of $\pi^N_k$.
Each family draws its common ancestor in generation $k-1$ uniformly without replacement.
Tracing back the ancestral lineage, we denote by $A_j(k,i)$ the\label{Conv:not:anc} label of the ancestor in generation $j$ of the individual $i$ of generation $k$, for $j\in\N_0$ with $j\leq k$.

We are interested in the genealogical distances between the individuals in each generation. Given a distance matrix $\rho^N_0\in\Uu_N$, we define $\rho^N_0(i,j)$ as the genealogical distance between the individuals $i$ and $j$ in generation $0$, for $i,j\in\sN$. Then we define the genealogical distance\label{Conv:not:rhoNk} $\rho^N_\ell(i,j)$ between individuals $i$ and $j$ in a later generation $\ell\in\N$ by
\begin{equation*}
\rho^N_\ell(i,j)=\left\{
\begin{aligned}
&2c_N(\ell-\max\{k=0,\ldots,\ell:A_k(\ell,i)=A_k(\ell,j)\})\quad
\text{if }A_0(\ell,i)=A_0(\ell,j)\\
&2c_N\ell+\rho^N_0(A_0(\ell,i),A_0(\ell,j))\quad\text{else,}
\end{aligned}
\right.
\end{equation*}
where we choose the scaling factor
\begin{equation}
\label{Conv:eq:cN}
c_N=\int\sum_{i=1}^Nx(i)\frac{Nx(i)-1}{N-1}\Xi^N(dx).
\end{equation}
We always assume $c_N>0$. Analogously to e.\,g.\ \cites{GPW13,Sampl}, the genealogical distance $\rho^N_\ell(i,j)$ between the individuals $i$ and $j$ in generation $\ell$ is, up to the scaling factor, the number of generations backwards until these individuals have the same ancestor if they have the same ancestor in generation $0$, else $\rho^N_\ell(i,j)$ is given by the genealogical distance of their ancestors in generation $0$.
The quantity $c_N$ is known as the pairwise coalescence probability, i.\,e.\ the probability that two individuals that are sampled uniformly without replacement from some generation $k\in\N$ have the same ancestor in generation $k-1$.
Indeed, conditionally given $x^N_k$, the first sample is in any family with probability $x^N_k(i)$ when $Nx^N_k(i)$ is the size of that family. Conditionally given $x^N_k$ and the first sample, the second sample is in the same family with probability $(Nx^N_k(i)-1)/(N-1)$.

In the next subsections, we state four invariance principles for processes that we read off from this population model. The invariance principles define the limit processes, but we recall the limit processes independently in Section~\ref{Conv:sec:TVFV}. All our limit processes are characterized by the probability measures on the simplex $\Delta$, we denote the space of these measures by ${\Mm_1(\Delta)}$.\label{Conv:not:MfD} If $\Xi\in{\Mm_1(\Delta)}$ satisfies the condition
\begin{equation}
\label{Conv:eq:nd}
\Xi\{0\}>0\text{ or } \int\left|x\right|_1\left|x\right|_2^{-2}\Xi(dx)=\infty,
\end{equation}
then we speak of the dust-free case and we write $\Xi\in\nd$. The converse case is called the case with dust. We set $\dust={\Mm_1(\Delta)}\setminus\nd$. For $\Xi\in{\Mm_1(\Delta)}$, we denote by $\Xi_0$ the measure on $\Delta$ with
\begin{equation}
\label{Conv:eq:Xi-dec}
\Xi=\Xi_0+\Xi\{0\}\delta_0.
\end{equation}

\subsection{Processes with values in the space of metric measure spaces}
\label{Conv:sec:inv-mm}
Let $\chi_0^N\in\U_N$, and let $\rho^N_0$ be a random variable with distribution $\nu^{N,\chi^N_0}$ that is independent of the sequence $(\pi^N_k,k\in\N)$ from Subsection \ref{Conv:sec:Cannings}.
We define the process $(\rho^N_k,k\in\N_0)$ as in Subsection~\ref{Conv:sec:Cannings} from $(\pi^N_k,k\in\N)$ and the initial state $\rho^N_0$.
For $k\in\N_0$, we set
\begin{equation}
\label{Conv:eq:rho-chi}
\chi^N_k=\ew{\sN,\rho^N_k,N^{-1}\sum_{i=1}^N\delta_i}=\psi_N(\rho^N_k).
\end{equation}
While $\rho^N_k$ describes the genealogy of generation $k$ as a leaf-labeled tree, the unlabeled tree is given by $\chi^N_k$.
We call the process $(\chi^N_k,k\in\N_0)$ a $\U_N$-valued $\Xi^N$-Cannings chain with initial state $\chi^N_0$. We call this process also a tree-valued $\Xi^N$-Cannings chain.
\begin{rem}[Markov property and transition kernel]
\label{Conv:rem:Markov-chi}
We denote by $\p_N$ the set of partitions of $\sN$. For $\pi\in\p_N$ and $i\in\sN$, we denote by $\pi(i)$ the integer $k$ such that the $k$-th block of $\pi$ contains $i$ when the blocks are ordered increasingly according to their respective
smallest element. As in \cite{Sampl}, we associate with each element $\pi$ of $\p_N$ a transformation $\Uu_N\to\Uu_N$, which we also denote by $\pi$, by
\begin{equation}
\label{Conv:eq:pn-Un}
\pi(\rho)=(\rho(\pi(i),\pi(j)))_{i,j\in\sn}.
\end{equation}
We write $\underline{\underline 2}_N=(\I{i\neq j})_{i,j\in\sN}$.
The Markov property of $(\chi^N_k,k\in\N_0)$ follows as for each $k\in\N_0$, the distance matrix
$\rho^N_{k+1}-c_N\underline{\underline 2}_N$
has conditional distribution $\pi^N_{k+1}(\nu^{N,\chi^N_k})$ given $\pi^N_{k+1}$ and $(\rho^N_j,j\leq k)$ by the construction in Section~\ref{Conv:sec:Cannings}. We denote by $p_N$ the transition kernel of $(\chi^N_k,k\in\N_0)$ which can be stated as
\[p_N(\chi,B)=\sum_{\pi\in\p_N}\P(\pi^N_1=\pi)
\int\nu^{N,\chi}(d\rho)\I{\psi_N(\pi(\rho)+c_N\underline{\underline 2}_N)\in B}\]
for all $\chi\in\U_N$ and measurable $B\subset\U$. This transition kernel generalizes the one of a tree-valued Moran model from \cite{GPW13} or a tree-valued $\Lambda$-Cannings process that is discussed in \cite{KL14}.
\end{rem}
For $c\in\R_+$, we set
$\Delta_c=\{x\in\Delta:x(1)>c\}$.
\begin{thm}
\label{Conv:thm:nd}
Let $\Xi\in\nd$. Assume that $\chi^N_0$ converges to some $\chi_0$ in $(\U,\dGP)$ as $N$ tends to infinity. Furthermore, assume that
\begin{equation}
\label{Conv:eq:cond-cN}
\lim_{N\to\infty}c_N=0
\end{equation}
and that
\begin{equation}
\label{Conv:eq:cond-conv}
\begin{aligned}
&\text{for arbitrarily small $\ep>0$, on $\Delta_\ep$, the finite measures $c_N^{-1}\Xi^N(dx)$}\\
&\text{converge weakly to $\left|x\right|_2^{-2}\Xi(dx)$ as $N$ tends to infinity.}
\end{aligned}
\end{equation}
Then the processes $(\chi^N_{\fl{c_N^{-1}t}},t\in\R_+)$ converge in distribution to a $\U$-valued $\Xi$-Fleming-Viot process with initial state $\chi_0$ in the space of càdlàg paths in $(\U,\dGP)$, endowed with the Skorohod metric, as $N$ tends to infinity.
\end{thm}
Condition \reff{Conv:eq:cond-cN} is Condition (1.6) in \cite{Sag03} and ensures that the limit process is a process in continuous time with no fixed jumps. Condition \reff{Conv:eq:cond-conv} is Condition (2.9) of \cite{Sag03} and yields the convergence of the transition kernels to the generator of the limit process.
The assumption $\Xi\in\nd$ in the theorem above is necessary as the $\U$-valued $\Xi$-Fleming-Viot process exists only for these $\Xi$. We include the case with dust in two ways: In Section~\ref{Conv:sec:inv-distr}, we consider the distance matrix distributions. In Section~\ref{Conv:sec:dec}, we decompose the genealogical trees.

\subsection{Processes of distance matrix distributions}
\label{Conv:sec:inv-distr}
We associate with the $\U_N$-valued $\Xi^N$-Canning chain from the last subsection the process
\[(\xi^N_k,k\in\N_0)=(\nu^{\chi^N_k},k\in\N_0)\]
with values in the space $(\UUerg,\dP)$ which we recalled in Section \ref{Conv:sec:prelim:mm}. Also $(\xi^N_k,k\in\N_0)$ is a Markov process. This follows from the Markov property of $(\chi^N_k,k\in\N_0)$ as $\xi^N_k$ uniquely determines $\chi^N_k$ by the Gromov reconstruction theorem. As $(\xi^N_k,k\in\N_0)$ takes values in the space of distance matrix distributions
\[\UU_N=\{\nu^\chi:\chi\in\U_N\}\subset\UUerg,\]
we call this process a $\UU_N$-valued $\Xi^N$-Cannings chain. When we do not want to specify the state space, we call also this process a tree-valued $\Xi^N$-Cannings chain.
\begin{thm}
\label{Conv:thm:distr}
Let $\Xi\in{\Mm_1(\Delta)}$. Assume that conditions~\reff{Conv:eq:cond-cN} and \reff{Conv:eq:cond-conv} hold, and that $\xi^N_0$ converges weakly to some probability measure $\xi_0$ on $\Uu$ as $N$ tends to infinity. Then the processes
$(\xi^N_{\fl{c_N^{-1}t}},t\in\R_+)$
converge in distribution to a $\UUerg$-valued $\Xi$-Fleming-Viot process with initial state $\xi_0$ in the space of càdlàg paths in $(\UUerg,\dP)$, endowed with the Skorohod metric, as $N$ tends to infinity.
\end{thm}
\begin{rem}
Consider the case that $\Xi\in\nd$ and that $\xi_0=\nu^{\chi_0}$ for some $\chi_0\in\U$. Let $(\chi_t,t\in\R_+)$ be a $\U$-valued $\Xi$-Fleming-Viot process with initial state $\chi_0$. Then, by the definition in \cite{Sampl}*{Section 8}, the process $(\nu^{\chi_t},t\in\R_+)$ is a $\UUerg$-valued $\Xi$-Fleming-Viot process and the assertion of Theorem~\ref{Conv:thm:distr} follows from Theorem~\ref{Conv:thm:nd} as the map $\U\to\UUerg$, $\chi\mapsto\nu^\chi$ is continuous.
\end{rem}

\subsection{Processes with values in the space of marked metric measure spaces}
\label{Conv:sec:dec}
Recall the decomposition $\beta:\Uu_N\to\hat\Uu_N$ of ultrametric distance matrices from Section~\ref{Conv:sec:dist-dec} and the construction $\hat\psi_N:\hat\Uu_N\to\hat\U_N$ of marked metric measure spaces from marked distance matrices from Section~\ref{Conv:sec:prelim:mm}. We define a process $(\hat\chi^N_k,k\in\N)$ which we call a $\hat\U_N$-valued $\Xi^N$-Cannings chain by\label{Conv:not:hatchiNk}
\[\hat\chi^N_k=\hat\psi_N\circ\beta(\rho^N_k)\]
for $k\in\N_0$, where $(\rho^N_k,k\in\N_0)$ is defined as in Section~\ref{Conv:sec:inv-mm}. More loosely, we also call the process $(\hat\chi^N_k,k\in\N)$ a tree-valued $\Xi^N$-Cannings chain.

We denote by $b_N$ the probability that an individual that is sampled uniformly from a fixed generation belongs to a family with more than one member. By construction,\label{Conv:not:bN}
\[b_N=\E\left[\sum_{i=1}^\infty x^N_1(i)\I{x^N_1(i)>1/N}\right].\]

\begin{thm}
\label{Conv:thm:dust}
Let $\Xi\in{\Mm_1(\Delta)}$. Assume that conditions~\reff{Conv:eq:cond-cN} and \reff{Conv:eq:cond-conv} hold, and that $\hat\chi^N_0$ converges to some $\hat\chi_0$ in $(\hat\U,\dMGP)$ as $N$ tends to infinity.
\begin{enumerate}[label=(\roman*),ref=(\roman*)]
\item If $\Xi\in\dust$, assume in addition that
\begin{equation}
\label{Conv:eq:singl-dust}
\lim_{N\to\infty}b_N/c_N=\int\left|x\right|_1\left|x\right|_2^{-2}\Xi(dx).
\end{equation}
\item\label{Conv:item:thm:dust:nd} If $\Xi\in\nd$, assume in addition that the marked metric measure space $\hat\chi_0$ supports only the zero mark.
\end{enumerate}
Then the processes $(\hat\chi^N_{\fl{c_N^{-1}t}},t\in\R_+)$ converge in distribution to a $\hat\U$-valued $\Xi$-Fleming-Viot process with initial state $\hat\chi_0$ in the space of càdlàg paths in $(\hat\U,\dMGP)$, endowed with the Skorohod metric, as $N$ tends to infinity.
\end{thm}
In the dust-free case, assumption \ref{Conv:item:thm:dust:nd} is necessary for right continuity of the limit process at time $0$.
The expression on the right-hand side of \reff{Conv:eq:singl-dust} is the rate $\lambda_{1,\{\{1\}\}}$ in the martingale problem for the $\hat\U$-valued $\Xi$-Fleming-Viot process which we recall in Section \ref{Conv:sec:TVFV-dust}. In the lookdown construction from \cite{Sampl}, $\lambda_{1,\{\{1\}\}}$ is the rate of reproduction events in which the individual on a fixed level belongs to a non-singleton family, which we will use in Remark \ref{Conv:rem:nec-mod}.
The expression on the right-hand side of \reff{Conv:eq:singl-dust} also occurs as the parameter of the limiting exponential distribution of the length of an external branch that is drawn randomly from a $\Xi$-coalescent as the sample size tends to infinity, see \cite{M10}.

We also decompose some metric measure spaces analogously to Section~\ref{Conv:sec:dist-dec}: We denote also by $\beta$ the function\label{Conv:not:betamm} $\U_N\to\hat\U_N$ that maps $\chi$ to $\psi_N\circ\beta(\rho)$ where $\rho$ is any element of $\Uu_N$ with $\psi_N(\rho)=\chi$.
The map $\beta:\U_N\to\hat\U_N$ decomposes the (unlabeled) tree given by (the isomorphy class of) a metric measure space $\chi$ such that in $\beta(\chi)$, the lengths of the external branches are encoded by the marks, and the distances between their starting points are given by the metric.
Note that
$\hat\chi^N_k=\beta(\chi^N_k)$ for all $k\in\N_0$.
The Markov property of the process $(\hat\chi^N_k,k\in\N_0)$ now follows from the Markov property of $(\chi^N_k,k\in\N_0)$ as $\chi^N_k$ is uniquely determined by $\beta(\chi^N_k)$.

\subsection{Another decomposition of the evolving genealogical trees}
\label{Conv:sec:mod}
In Theorem~\ref{Conv:thm:dust}, the external branches of genealogical trees play a crucial role.
However, for the proof in the case with dust, it seems more convenient to work with a different decomposition of the genealogical trees that does not bring about the freeing phenomenon which we mention below. We will therefore use Proposition~\ref{Conv:prop:mod} below in the proof of Theorem~\ref{Conv:thm:dust} in the case with dust. The process of these differently decomposed trees corresponds to the construction in \cite{Sampl}*{Section 7} where an infinite population is considered.

First we define a process $((r^N_k,v^N_k),k\in\N_0)$ of marked distance matrices.
Let $\tilde\chi^N_0\in\hat\U_N$, let $(r^N_0,v^N_0)$ be a $\hat\Uu_N$-valued random variable with distribution $\nu^{N,\tilde\chi^N_0}$, and let $\rho^N_0=\alpha(r^N_0,v^N_0)$, with $\alpha$ defined in Section~\ref{Conv:sec:dist-dec}. We define the process $(\rho^N_k,k\in\N)$ as in Subsection~\ref{Conv:sec:Cannings} from the initial state $\rho_0$ and the sequence $(\pi^N_k,k\in\N)$, assuming that $(\pi^N_k,k\in\N)$ is independent of $(r^N_0,v^N_0)$.

For $\ell\in\N$, we define $(r^N_\ell,v^N_\ell)$ as\label{Conv:not:rvNk} follows. For $i\in\sN$, if there exists a latest generation $k\in[\ell]$ in which the individual $A_k(\ell,i)$ is in a non-singleton block of $\pi^N_k$, we set $v^N_\ell(i)=c_N(\ell-k+1)$. Else, that is, if $A_k(\ell,i)$ forms a singleton block of $\pi^N_k$ for each $k\in[\ell]$, then we set $v^N_\ell(i)=c_N\ell+v^N_0(A_0(\ell,i))$.
For $i,j\in\sN$, we set
\[r^N_\ell(i,j)=(\rho^N_\ell(i,j)-v^N_\ell(i)-v^N_\ell(j))\I{i\neq j}.\]
\begin{lem}
It holds $(r^N_k,v^N_k)\in\hat\Uu_N$ for all $k\in\N_0$. 
\end{lem}
\begin{proof}
For $k\in\N$, we denote by $\cup\sigma^N_k$ the union of the non-singleton blocks of $\pi^N_k$ (anticipating the notation of Remark \ref{Conv:rem:Markov-tilde} below). By construction,
\begin{align*}
r^N_k(i,j)
=&v^N_{k-1}(A_{k-1}(k,i))\I{i\in\cup\sigma^N_k}
+r^N_{k-1}(A_{k-1}(k,i),A_{k-1}(k,j))\\
&+v^N_{k-1}(A_{k-1}(k,j))\I{j\in\cup\sigma^N_k}
\end{align*}
for all $k\in\N$ and all distinct $i,j\in[N]$. Under the assumption that $r^N_{k-1}$ satisfies the triangle inequality, it is easily checked that also $r^N_k$ satisfies the triangle inequality. The assertion follows by induction.
\end{proof}
Analogously to \cite{Sampl}*{Remark 7.1}, the quantity $v^N_k(i)$ can be interpreted as the age of the individual $i$ of generation $k$. The quantity $v^N_k(i)$ needs not coincide with the (scaled) length of the external branch that ends in individual $i$ in the (scaled) genealogical tree of generation $k$. Indeed, if $v_k(i)<k$ and no other members of the family of $A_{k-v^N_k(i)}(k,i)$ in generation $k-v^N_k(i)$ have descendants in generation $k$, then this external branch is longer than $v^N_k(i)$. A related phenomenon in evolving coalescents is the so-called freeing where internal branches become part of external branches (see Dahmer and Kersting \cite{DK14}, in particular Figure 2 therein).
The quantity $r^N_k(i,j)$ is the distance between the individuals that correspond to the parents in \cite{Pathw}.

We obtain a $\hat\U_N$-valued process $(\tilde\chi^N_k,k\in\N_0)$ by setting\label{Conv:not:tildechiNk} $\tilde\chi^N_k=\hat\psi_N(r^N_k,v^N_k)$ for $k\in\N$. The definition of $(r^N_0,v^N_0)$ yields $\tilde\chi^N_0=\hat\psi_N(r^N_0,v^N_0)$.
\begin{rem}[Markov property and transition kernel]
\label{Conv:rem:Markov-tilde}
We denote by $\S_N$ the set of semi-partitions\label{Conv:not:Sn} of $\sN$. A semi-partition $\sigma$ of $\sN$ is a system of nonempty disjoint subsets of $\sN$ which we call blocks. We denote by $\cup\sigma$ the union of the blocks of $\sigma$ (which can be a proper subset of $[N]$). For $\sigma\in\S_N$ and $i\in\sN$, we define $\sigma(i)=\pi(i)$ where $\pi$ is the partition of $\sn$ that has the same non-singleton blocks as $\sigma$, that is, $\{B\in\sigma:\#B\geq 2\}=\{B\in\pi:\#B\geq 2\}$, and $\pi(i)$ is defined as in Remark~\ref{Conv:rem:Markov-chi}.
As in \cite{Sampl}, we associate with each element $\sigma$ of $\S_N$ a transformation\label{Conv:not:Sn-transf} $\hat\Uu_N\to\hat\Uu_N$, which we also denote by $\sigma$, by $\sigma(r,v)=(r',v')$, where
\[v'(i)=v(\sigma(i))\I{i\notin\cup\sigma}\]
and
\[r'(i,j)=(v(\sigma(i))\I{i\in\cup\sigma}+r(\sigma(i),\sigma(j))+v(\sigma(j))\I{j\in\cup\sigma})\I{i\neq j}\]
for $i,j\in\sN$.

For $k\in\N$, let $\sigma^N_k$ be the semi-partition that consists of the non-singleton blocks of the partition $\pi^N_k$ from Section~\ref{Conv:sec:Cannings}, $\sigma^N_k=\{B\in\pi^N_k:\#B\geq 2\}$.
We write $\underline 1_N=(1)_{i\in\sN}$.
The Markov property of $(\tilde\chi^N_k,k\in\N_0)$ follows as by construction, for each $k\in\N_0$, the marked distance matrix $(r^N_{k+1},v^N_{k+1}-c_N\underline 1_N)$ has conditional distribution $\sigma^N_{k+1}(\nu^{N,\tilde\chi^N_k})$ given $\sigma^N_{k+1}$ and $((r^N_j,v^N_j),j\leq k)$. We denote by $\tilde p_N$ the transition kernel of $(\tilde\chi^N_k,k\in\N_0)$ which can be stated as
\[\tilde p_N(\chi,B)=\sum_{\sigma\in\S_N}\P(\sigma^N_1=\sigma)
\int\nu^{N,\chi}(d(r,v))\I{\hat\psi_N(\sigma(r,v)+c_N(0,\underline 1_N))\in B}\]
for all $\chi\in\hat\U_N$ and measurable $B\subset\hat\U$.
\end{rem}
\begin{prop}
\label{Conv:prop:mod}
Assume that assumptions of Theorem~\ref{Conv:thm:dust} hold and that $\tilde\chi^N_0$ converges to $\hat\chi_0$ in $(\hat\U,\dMGP)$ as $N$ tends to infinity. Then the processes $(\tilde\chi^N_{\fl{c_N^{-1}t}},t\in\R_+)$ converge in distribution to a $\hat\U$-valued $\Xi$-Fleming-Viot process with initial state $\hat\chi_0$ in the space of càdlàg paths in $(\hat\U,\dMGP)$, endowed with the Skorohod metric, as $N$ tends to infinity.
\end{prop}
\begin{rem}
\label{Conv:rem:nec-mod}
In Proposition~\ref{Conv:prop:mod}, the assumption~\reff{Conv:eq:singl-dust} is necessary in case $\Xi\in\dust$. To see this, let $t\in(0,\infty)$, and let $v_t=(v_t(i),i\in\N)$ be defined as in \cite{Sampl}*{Section 7.2}. By the lookdown construction in \cite{Sampl}, the truncated first entry $v_t(1)\wedge t$ is equal in distribution to $T\wedge t$ for an exponentially distributed random variable $T$ with parameter
$\int\left|x\right|_1\left|x\right|_2^{-2}\Xi_0(dx)$.
Moreover, by the construction in Subsection \ref{Conv:sec:Cannings}, the random variable $v^N_{\fl{c_N^{-1}t}}(1)\wedge(c_N\fl{c_N^{-1}t})$ is distributed as $c_N(T_N\wedge \fl{c_N^{-1}t})$ for a geometrically distributed random variable $T_N$ with parameter $b_N$. Let $(\hat\chi_s,s\in\R_+)$ be a $\hat\U$-valued $\Xi$-Fleming-Viot process with initial state $\hat\chi_0$ (as defined in \cites{Sampl,Pathw}) and let $\varpi:\hat\Uu\to\R_+$, $(r,v)\mapsto v(1)$. Then, as in \cite{Sampl}*{Proposition 4.7}, the random variable $v_t(1)$ has distribution $\E[\varpi(\nu^{\hat\chi_t})]$. Moreover, as $v^N_{\fl{c_N^{-1}t}}$ is exchangeable, the random variable $v^N_{\fl{c_N^{-1}t}}(1)$ has distribution $\E[\varpi(\nu^{\tilde\chi^N_{\fl{c_N^{-1}t}}})]$.

Now assume that $(\tilde\chi^N_{\fl{c_N^{-1}s}},s\in\R_+)$ converges as asserted in Proposition~\ref{Conv:prop:mod}. Then, as the path $(\hat\chi_s,s\in\R_+)$ is a.\,s.\ continuous at fixed times, also $\tilde\chi^{N}_{\fl{c_N^{-1}t}}$ converges in distribution to $\hat\chi_t$. By continuity of the map $\varpi$, also $v^N_{\fl{c_N^{-1}t}}(1)$ converges in distribution to $v_t(1)$. As $t$ can be chosen arbitrarily large, it follows that $c_N T_N$ converges in distribution to $T$ which implies condition~\reff{Conv:eq:singl-dust}. 
\end{rem}

\section{Tree-valued Fleming-Viot processes}
\label{Conv:sec:TVFV}
In this section, we recall from \cite{Sampl} the martingale problems for tree-valued $\Xi$-Fleming-Viot processes. Path regularity is shown in \cite{Pathw}.
(In \cites{Sampl,Pathw}, tree-valued $\Xi$-Fleming-Viot processes are considered for finite measures $\Xi$ on $\Delta$.)
For $n\in\N$, we define the restrictions\label{Conv:not:gamma:dm}
\[\gamma_n:\Uu\cup\bigcup_{\ell\geq n}\Uu_\ell\to\Uu_n,\quad \rho\mapsto(\rho(i,j))_{i,j\in\sn}\]
and
\[\gamma_n:\hat\Uu\cup\bigcup_{\ell\geq n}\hat\Uu_\ell\to\hat\Uu_n,\quad
(r,v)\mapsto((r(i,j))_{i,j\in\sn},(v(i)	)_{i\in\sn}).\]
Let\label{Conv:not:Cn} $\C_n$ denote the set of bounded differentiable functions $\R^{n^2}\to\R$ with bounded uniformly continuous derivative. For $\phi\in\C_n$, we write also $\phi$ for the function $\phi\circ\gamma_n$. For $\phi\in\C_n$, we call the function $\U\to\R$, $\chi\mapsto\nu^\chi\phi$ the polynomial associated with $\phi$. As in \cite{Lohr13}*{Corollary 2.8}, the algebra of polynomials\label{Conv:not:Pi}
\[\Pi=\{\U\to\R,\chi\mapsto\nu^\chi\phi: n\in\N, \phi\in\C_n\}\] 
is convergence determining in $\U$.

Analogously, let\label{Conv:not:hatCn} $\hat\C_n$ be the set of bounded differentiable functions $\phi:\R^{n^2}\times\R^n\to\R$ with bounded uniformly continuous derivative. For $\phi\in\hat\C_n$, we denote also the function $\phi\circ\gamma_n$ by $\phi$, and we associate with $\phi$ the marked polynomial $\hat\U\to\R$, $\chi\mapsto\nu^\chi\phi$. The algebra of marked polynomials\label{Conv:not:hatPi}
\[\hat\Pi=\{\hat\U\to\R,\chi\mapsto\nu^\chi\phi: n\in\N, \phi\in\hat\C_n\}\]
is convergence determining in $\hat\U$, see \cite{Lohr13}*{Corollary 2.8}.
The algebra\label{Conv:not:Cc}
\[\Cc=\{\UUerg\to\R,\nu\mapsto\nu\phi:n\in\N,\phi\in\C_n\}\]
is convergence determining in $\UUerg$, see \cite{Sampl}*{Remark 4.5}.

\subsection{The \texorpdfstring{$\U$}{U}-valued \texorpdfstring{$\Xi$}{Xi}-Fleming-Viot process}
\label{Conv:sec:TVFV-nd}
A $\U$-valued $\Xi$-Fleming-Viot process exists for $\Xi\in\nd$ and any initial state in $\U$. It is a Markov process and has a version with càdlàg paths and no discontinuities at fixed times.
This process is the unique solution of the martingale problem $(B,\Pi)$, we recall the generator $B$ in this subsection. For $n\in\N$, let $\p_n$ be the set of partitions of $\sn$, and let $\kappa^\infty_n$ be the probability kernel from $\Delta$ to $\p_n$ given by Kingman's correspondence, which we also recall here as we need it in Section~\ref{Conv:sec:proof-conv-rates}.
Consider independent uniform $[0,1]$-valued random variables $U_1,\ldots,U_n$. For $x\in\Delta$, let $\kappa^\infty_n(x,\cdot)$ be the distribution of the random partition in $\p_n$ such that two different integers $i$ and $j$ are in the same block if and only if there exists $\ell\in\N$ with
\[U_i,U_j\in\left(\sum_{k=1}^{\ell-1} x(k),\sum_{k=1}^{\ell} x(k)\right).\]

For $\phi\in\C_n$, we define the function
\[\langle \nabla\phi, \underline{\underline 2} \rangle:\R^\N\to\R,\quad
\rho\mapsto 2\sum_{\substack{i,j\in\N\\ i\neq j}} \frac{\partial}{\partial \rho(i,j)}\phi(\rho).\]

Let $\mathbf{0}_n=\{\{1\},\ldots,\{n\}\}$, and recall $\Xi_0$ from equation \reff{Conv:eq:Xi-dec}. For $\pi\in\p_n\setminus\{\mathbf{0}_n\}$, we define\label{Conv:not:lambdapi}
\begin{align*}
\lambda_\pi=&\int\kappa^\infty_n(x,\pi)\left|x\right|_2^{-2}\Xi_0(dx)\\
&+\Xi\{0\}\I{\pi\text{ contains one doubleton and apart from that only singletons}}.
\end{align*}
The rates $\lambda_\pi$ are equal to those of Schweinsberg \cite{Schw00}, see \cite{Sampl}*{Remark 6.1}. Moreover, we associate with each element of $\p_n$ a transformation on $\Uu_n$ as in Remark~\ref{Conv:rem:Markov-chi}.

Now we set $B=B_{\rm grow}+B_{\rm res}$ with
\[B_{\rm grow}\Phi(\chi)=\int\nu^\chi(d\rho)\langle \nabla \phi, \underline{\underline 2} \rangle(\rho)\]
and
\[B_{\rm res}\Phi(\chi)=\sum_{\pi\in\p_n\setminus\{\mathbf{0}_n\}}\lambda_\pi\int\nu^\chi(d\rho)(\phi(\pi(\gamma_n(\rho)))-\phi(\rho))\]
for $\chi\in\U$, $n\in\N$, and $\phi\in\C_n$ with associated polynomial $\Phi$.

\subsection{The \texorpdfstring{$\mathbb{\hat U}$}{U}-valued \texorpdfstring{$\Xi$}{Xi}-Fleming-Viot process}
\label{Conv:sec:TVFV-dust}
We first consider the case $\Xi\in\nd$.
Using the isometry\label{Conv:not:beta0}
\[\beta_0:\U\to\hat\U,\quad\ew{X,\rho,\mu}\mapsto\ew{X,\rho,\mu\otimes\delta_0}\]
which adds a mark component that is concentrated in zero,
we define for $\chi\in\U$ a $\hat\U$-valued $\Xi$-Fleming-Viot process with initial state $\beta_0(\chi)$ by $(\beta_0(\chi_t),t\in\R_+)$, where $(\chi_t,t\in\R_+)$ is a $\U$-valued $\Xi$-Fleming-Viot process with initial state $\chi$ and càdlàg paths.
Also this process is Markov and has càdlàg paths with no discontinuities at fixed times.

Now we consider the case $\Xi\in\dust$. Then for each initial state in $\hat\U$, there exists a $\hat\U$-valued $\Xi$-Fleming-Viot process which is a Markov process and has a version with càdlàg paths and no discontinuities at fixed times. This process is the unique solution of the martingale problem $(\hat B,\hat\Pi)$, we recall the generator $\hat B$ now. For $n\in\N$, we define the set $\S_n$ of semi-partitions of $\sn$ and the transformation on $\hat\Uu_n$ associated with each element of $\S_n$ as in Remark~\ref{Conv:rem:Markov-tilde}.

Let $K^\infty_n$ be the probability kernel from $\Delta$ to $\S_n$ defined as follows: Consider independent uniform $[0,1]$-valued random variables $U_1,\ldots,U_n$. For $x\in\Delta$, let $K^\infty_n(x,\cdot)$ be the distribution of the random element $\sigma$ of $\S_n$ such that any two integers $i$ and $j$ are in a common subset $B\in\sigma$ if and only if there exists $\ell\in\N$ with
\[U_i,U_j\in\left(\sum_{k=1}^{\ell-1} x(k),\sum_{k=1}^{\ell} x(k)\right).\]

For $\phi\in\hat\C_n$, we define the function
\[\langle \nabla^{v} \phi, \underline 1 \rangle:\R^{\N^2}\times\R^\N\to\R,\quad
(r,v)\mapsto\sum_{i\in\N} \frac{\partial}{\partial v(i)}\phi(r,v).\]
For $\sigma\in\S_n\setminus\{\emptyset\}$, we define the rates\label{Conv:not:lambdansigma}
\[\lambda_{n,\sigma}=\int K^\infty_n(x,\sigma)\left|x\right|_2^{-2}\Xi_0(dx).\]
These rates are finite by the assumption that $\Xi\in\dust$ and as
$K^\infty_n(x,\sigma)\leq K^\infty_1(x,\{\{1\}\})=|x|_1$
for all $x\in\Delta$. Now we set
\[\hat B=\hat B_{\rm grow}+\hat B_{\rm res},\]
\[\hat B_{\rm grow}\Phi(\chi)=\int\nu^{\chi}(dr\,dv)\langle \nabla^{v} \phi, \underline 1 \rangle(r,v)\]
and
\[\hat B_{\rm res}\Phi(\chi)=\sum_{\sigma\in\S_n\setminus\{\emptyset\}}\lambda_{n,\sigma}\int\nu^{{\chi}}(dr\,dv)(\phi(\sigma(\gamma_n(r,v)))-\phi(r,v))\]
for ${{\chi}}\in\hat\U$, $n\in\N$, and $\phi\in\hat\C_n$ with associated marked polynomial $\Phi$.

\subsection{The \texorpdfstring{$\UUerg$}{U}-valued \texorpdfstring{$\Xi$}{Xi}-Fleming-Viot process}
\label{Conv:sec:TVFV-dmd}
A $\UUerg$-valued $\Xi$-Fleming-Viot process exists for every $\Xi\in{\Mm_1(\Delta)}$ and every initial state in $\UUerg$. It is a Markov process and has a version with càdlàg paths and no discontinuities at fixed times. This process is the unique solution of the martingale problem $(C,\Cc)$ which is defined as follows.
Let the rates $\lambda_\pi$ for $\pi\in\p_n\setminus\{\mathbf{0}_n\}$, $n\in\N$ be defined from the measure $\Xi\in{\Mm_1(\Delta)}$ as in Subsection~\ref{Conv:sec:TVFV-nd}. We define the generator $C$ by
\[C=C_{\rm grow}+C_{\rm res}\]
\[C_{\rm grow}\Psi(\xi)
=\int\xi(d\rho)\langle\nabla\phi,\underline{\underline{2}}\rangle(\rho)\]
\[C_{\rm res}\Psi(\xi)=\sum_{\pi\in\p_n\setminus\{\mathbf{0}_n\}}
\lambda_\pi\int\xi(d\rho)(\phi(\pi(\gamma_n(\rho)))-\phi(\rho))\]
for $\xi\in\UUerg$, $n\in\N$, $\phi\in\C_n$, and $\Psi\in\Cc$, $\Psi:\xi'\mapsto\xi'\phi$.

If $(\hat\chi_t,t\in\R_+)$ is a $\hat\U$-valued $\Xi$-Fleming-Viot process, then the process $(\alpha(\nu^{\hat\chi_t}),t\in\R_+)$ of the associated distance matrix distributions is a $\UUerg$-valued $\Xi$-Fleming-Viot process.

\section{An example}
\label{Conv:sec:counterex}
In this section, we consider a sequence of Cannings models that does not satisfy the assertion of Theorem~\ref{Conv:thm:dust}, but all its assumptions for $\Xi\in\dust$ except condition~\reff{Conv:eq:singl-dust}.
The convergence of the $\hat\U_N$-valued Cannings chains to the $\hat\U$-valued $\Xi$-Fleming-Viot process in Theorem~\ref{Conv:thm:dust} is excluded as the length of a randomly sampled external branch converges to zero in distribution, while this quantity is a.\,s.\ non-zero at a fixed time for a $\hat\U$-valued Fleming-Viot process in the case with dust. The convergence in Theorem \ref{Conv:thm:distr} nevertheless holds. Here it comes to bear that the length of the first external branch is not a continuous functional of the infinite ultrametric that describes the genealogy.
In the example of the present section, the length of a randomly sampled external branch converges to zero due to ``perturbative'' reproduction events that occur at high rate. However, these reproduction events are not visible in the limiting genealogy as each of them affects only a small part of the population. 

We now work in the context of Section~\ref{Conv:sec:Cannings}. We assume in this section that for $N$ sufficiently large, the law $\Xi^N$ of the sequence of the family sizes $x=(x(i),i\in\N)$ in a reproduction event is given by the outcome of a two-step random experiment that can be described as follows. In the first step, we draw the type of the reproduction event from the set $\{\text{ordinary}, \text{perturbative}, \text{trivial}\}$. With probability $N^{-1}$, a reproduction event shall be ordinary, with probability $N^{-1/2}$ perturbative, and with probability $1-N^{-1}-N^{-1/2}$ trivial. In the second step, the sequence $x$ of the family sizes is drawn depending on the type of the reproduction event: 
\begin{enumerate}[label=(\roman*),ref=(\roman*)]
\item\label{Conv:item:ex-large-ord} If the reproduction event is ordinary, sample the largest family size $Nx(1)$ according to the binomial distribution with number of trials $N$ and success probability $1/2$. That is, each individual belongs to the largest family independently with probability $1/2$. Let the other families be singletons, so that $Nx(i)\in\{0,1\}$ for $i\geq 2$.
\item\label{Conv:item:ex-large-pert} If the reproduction event is perturbative, sample the largest family size $Nx(1)$ according to the binomial distribution with parameters $N$ and $N^{-1/3}$. Let the other families be singletons, so that $Nx(i)\in\{0,1\}$ for $i\geq 2$.
\item If the reproduction event is trivial, let all families be singletons, so that $Nx(i)=1$ for $i\in\sN$, and $Nx(i)=0$ for $i>N$. 
\end{enumerate}
The pairwise coalescence probability for this reproduction law equals
\[c_N=N^{-1}(\tfrac{1}{2})^2+N^{-1/2}(N^{-1/3})^2\]
and is asymptotically equivalent to $\frac{1}{4}N^{-1}$ as $N$ tends to infinity. The probability $b_N$ is bounded from below by
\[N^{-1/2}N^{-1/3}(1-N^{-(1/3)\cdot(N-1)}).\]
This is the probability that a perturbative reproduction event occurs in which the individual labeled by $1$ belongs to a non-singleton family.
Condition~\reff{Conv:eq:cond-conv} is satisfied with $\Xi=\delta_{(1/2,0,0,\ldots)}\in\dust$.
As $b_N/c_N$ tends to infinity as $N$ tends to infinity, condition~\reff{Conv:eq:singl-dust} is not satisfied.

In this section, we fix $t\in(0,\infty)$.
\begin{prop}
\label{Conv:prop:counterex}
For $N\in\N$, let $(\hat\chi^N_k,k\in\N_0)$ be a $\hat\U$-valued Cannings chain defined as in Section~\ref{Conv:sec:dec} from the measure $\Xi^N$. Let $(\hat\chi_s,s\in\R_+)$ be a $\hat\U$-valued $\delta_{(1/2,0,0,\ldots)}$-Fleming-Viot process.
Then $\hat\chi^N_{\fl{c_N^{-1}t}}$ does not converge in distribution to $\hat\chi_t$ as $N$ tends to infinity.
\end{prop}
\begin{rem}
Proposition~\ref{Conv:prop:counterex} implies that the processes $(\hat\chi^N_{\fl{c_N^{-1}s}},s\in\R_+)$ do not converge in distribution in the Skorohod metric to $(\hat\chi_s,s\in\R_+)$ as $N$ tends to infinity. This follows as $(\hat\chi_s,s\in\R_+)$ does a.\,s.\ not jump at fixed times.
\end{rem}
To prove Proposition~\ref{Conv:prop:counterex}, we use the following lemma which states that for large $N$, in the leaf-labeled genealogical tree of generation $\fl{c_N^{-1}t}$, the length of the external branch that ends in individual $1$ is typically small.
\begin{lem}
\label{Conv:lem:claim-ext}
Let $(\rho^N_k,k\in\N_0)$ be defined as in Section~\ref{Conv:sec:results} from the measure $\Xi^N$. Let
$\ep\in(0,t)$ and
\[v^N=(v^N(1),\ldots,v^N(N))=\Upsilon(\rho^N_{\fl{c_N^{-1}t}}).\]
Then,
$\P(v^N(1)>\ep)< 4\ep$
for sufficiently large $N$.
\end{lem}
\begin{proof}[Proof of Proposition~\ref{Conv:prop:counterex}]
Let $\varpi:\hat\Uu\to\R_+$, $(r,v)\mapsto v(1)$. By exchangeability, the first entry $v^N(1)$ of the vector $v^N$ in Lemma~\ref{Conv:lem:claim-ext} has distribution $\E[\varpi(\nu^{\hat\chi^N_{\fl{c_N^{-1}t}}})]$. By Lemma~\ref{Conv:lem:claim-ext}, the random variable $v^N(1)$ converges to zero in probability.

To obtain a contradiction, we assume that $\hat\chi^N_{\fl{c_N^{-1}t}}$ converges in distribution to $\hat\chi_t$. Then as in Remark~\ref{Conv:rem:nec-mod}, also the probability distributions $\E[\varpi(\nu^{\hat\chi^N_{\fl{c_N^{-1}t}}})]$ converge weakly to $\E[\varpi(\nu^{\hat\chi_t})]$. But as in Remark~\ref{Conv:rem:nec-mod}, the measure $\E[\varpi(\nu^{\hat\chi_t})]$ is not the Dirac measure in zero.
\end{proof}
\begin{proof}[Proof of Lemma~\ref{Conv:lem:claim-ext}]
Let $A_N$ be the number of ancestors in generation $\fl{c_N^{-1}t}-\fl{c_N^{-1}\ep}$ of the individuals of generation $\fl{c_N^{-1}t}$. Then,
\begin{equation}
\label{Conv:eq:ext-bound}
\P(v^N(1)>\ep,A_N\geq N/2)
\leq(1-N^{-1/2}N^{-1/3}(1-(1-N^{-1/3})^{N/2-1}))^{\fl{c_N^{-1}\ep}}.
\end{equation}
We sketch a proof of the bound \reff{Conv:eq:ext-bound}. The number of ancestors in generation $k$ of the individuals of generation $\fl{c_N^{-1}t}$ is non-decreasing in $k$ for $k\leq\fl{c_N^{-1}t}$. On the event $\{v^N(1)>\ep,A_N\geq N/2\}$, no reproduction event in which the individual labeled by $A_k(\fl{c_N^{-1}t},1)$ and another one of these ancestors are in the same block occurs in any generation $k$ with $\fl{c_N^{-1}t}-\fl{c_N^{-1}\ep}< k \leq\fl{c_N^{-1}t}$. The right-hand side of inequality~\reff{Conv:eq:ext-bound} is the probability that in none of the generations $k$ with $\fl{c_N^{-1}t}-\fl{c_N^{-1}\ep}< k \leq\fl{c_N^{-1}t}$, the reproduction event is perturbative and individual $A_k(\fl{c_N^{-1}t},1)$ is in the same family as any other of the at least $N/2$ many ancestors of the individuals of generation $\fl{c_N^{-1}t}$.

For $N$ sufficiently large, $\fl{c_N^{-1}\ep}\geq 3\ep N$. As
\[\log(1-N^{-1/3})^{N/2-1}\leq -(N/2-1)N^{-1/3}\to -\infty\quad(N\to\infty),\]
the right hand side of inequality~\reff{Conv:eq:ext-bound} is bounded from above by $(1-N^{-5/6}/2)^{3\ep N}$ for $N$ sufficiently large. This expression converges to zero as $N$ tends to infinity.

It now suffices to show $\limsup_{N\to\infty}\P(A_N< N/2)<4\ep$. The event that no ordinary reproduction events occur between generations
$\fl{c_N^{-1}t}-\fl{c_N^{-1}\ep}$ and $\fl{c_N^{-1}t}$ has probability at least $(1-N^{-1})^{4\ep N}$. Note that on this event, the random variable $N-A_N$ is stochastically bounded from above by $\sum_{i=1}^X Y_i$ where $X,Y_1,Y_2,\ldots$ are independent random variables, $X$ is binomially distributed with parameters $\fl{4\ep N}$ and $N^{-1/2}$, and all $Y_i$ are binomially distributed with parameters $N$ and $N^{-1/3}$. Here $X$ corresponds to the number of perturbative reproduction events, and $Y_i$ corresponds to the decrease in the number of ancestors, backwards in time, in the $i$-th of these reproduction events. Then we have
\begin{align}
\label{Conv:eq:ex-ancestors}
\P(A_N<N/2)\leq 1-(1-N^{-1})^{4\ep N}+\P(\sum_{i=1}^X Y_i>N/2).
\end{align}
We estimate
\begin{align}
\label{Conv:eq:ex-chebychev}
\P(\sum_{i=1}^X Y_i>N/2)\leq\P(X>8\ep N^{1/2})
+8\ep N^{1/2}\P(Y_1>(8\ep N^{1/2})^{-1}N/2).
\end{align}
By the Chebychev inequality, the first summand is bounded from above by
\[(8\ep N^{1/2}- 4\ep N^{1/2})^{-2}4\ep N^{1/2},\]
and the second summand is bounded from above by
\[8\ep N^{1/2}((8\ep N^{1/2})^{-1}N/2-N^{2/3})^{-2}N^{2/3}.\]
Hence, the expression on the right-hand side of~\reff{Conv:eq:ex-chebychev} tends to zero, and the expression on the right-hand side of~\reff{Conv:eq:ex-ancestors} converges to
$1-\expp{-4\ep}< 4\ep$ as $N\to\infty$.
\end{proof}

\section{Convergence of the transition kernels}
\label{Conv:sec:conv-gen}
This section contains the proofs of Theorems~\ref{Conv:thm:nd} and~\ref{Conv:thm:distr}, and the proof for the case $\Xi\in\nd$ in Proposition~\ref{Conv:prop:mod}.
We write $\gamma_n$ also for the restriction maps\label{Conv:not:gamma:part} $\S_N\to\S_n$ and $\p_N\to\p_n$ for $n\leq N$ (that is, $\gamma_n(\sigma)=\{B\cap\sn:B\in\sigma\}\setminus\{\emptyset\}$). We define the rates $\lambda_\pi$ and $\lambda_{n,\sigma}$ from the measure $\Xi$ as in Section \ref{Conv:sec:TVFV}.
We will need the following lemmas.
\begin{lem}
\label{Conv:lem:conv-rates}
Let $\Xi\in{\Mm_1(\Delta)}$ and assume that condition~\reff{Conv:eq:cond-conv} holds. Then,
\[\lim_{N\to\infty}c_N^{-1}\P(\gamma_n(\pi^N_1)=\pi)=\lambda_\pi\]
for all $n\in\N$ and $\pi\in\p_n\setminus\{\mathbf{0}_n\}$.
\end{lem}
\begin{lem}
\label{Conv:lem:conv-rates-dust}
Let $\Xi\in\dust$. Assume that conditions~\reff{Conv:eq:cond-conv} and~\reff{Conv:eq:singl-dust} hold. Then,
\[\lim_{N\to\infty}c_N^{-1}\P(\gamma_n(\sigma^N_1)=\sigma)=\lambda_{n,\sigma}\]
for all $n\in\N$ and $\sigma\in\S_n\setminus\{\emptyset\}$. 
\end{lem}
\begin{rem}
Note that the assumption of Lemma~\ref{Conv:lem:conv-rates} is condition (2.9) in \cite{Sag03}*{Theorem 2.1}, and that the limits in the assertion of Lemma~\ref{Conv:lem:conv-rates} are the limits (16) in \cite{MS01}*{Theorem 2.1}.
In Subsection \ref{Conv:sec:proof-conv-rates}, we prove Lemmas~\ref{Conv:lem:conv-rates} and~\ref{Conv:lem:conv-rates-dust} directly, using continuity properties in particular of the probability kernel associated with Kingman's correspondence.
\end{rem}

\subsection{Proofs of invariance principles}
\begin{proof}[Proof of Theorem~\ref{Conv:thm:nd}]
Let $n\in\N$, $N\geq n$, and $\phi\in\C_n$. As in \cite{GPW13}, we associate with $\phi$ not only the polynomial $\Phi:\U\to\R$, but also the $N$-polynomial
\[\Phi_N:\U_N\to\R,\quad\chi\mapsto\nu^{N,\chi}\phi.\]
As in Remark \ref{Conv:rem:Markov-chi}, let $\p_N$ denote the transition kernel of the Markov chain $(\chi^N_k,k\in\N_0)$.
Then,
\begin{align*}
p_N\Phi_N(\chi^N_0)&=\E[\Phi_N(\chi^N_1)]=\E[\phi(\rho^N_1)]\\
&=\sum_{\pi\in\p_n\setminus\{\mathbf{0}_n\}}\P(\gamma_n(\pi^N_1)=\pi)
\int\nu^{N,\chi^N_0}(d\rho)\phi(\pi(\gamma_n(\rho))+c_N\utwo_n).
\end{align*}
For the second equality, we use that $\nu^{N,\chi^N_k}$ is the conditional distribution of $\rho^N_1$ given $\chi^N_1$ which follows as $\rho^N_1$ is exchangeable.
By the construction in Section~\ref{Conv:sec:Cannings}, the conditional distribution of $\gamma_n(\rho^N_1)-c_N\underline{\underline{2}}_n$ given $\gamma_n(\pi^N_1)$ equals $\gamma_n(\pi^N_1)(\gamma_n(\nu^{N,\chi^N_0}))$. This yields the third equality. Now we have
\begin{align}
&c_N^{-1}(p_N-I)\Phi_N(\chi)\notag\\
&=\P(\gamma_n(\pi^N_1)=\mathbf{0}_n)
\int\nu^{N,\chi}(d\rho)c_N^{-1}(\phi(\rho+c_N\underline{\underline 2}_N)-\phi(\rho))\notag\\
&\quad+\sum_{\pi\in\p_n\setminus\{\mathbf{0}_n\}}c_N^{-1}\P(\gamma_n(\pi^N_1)=\pi)
\int\nu^{N,\chi}(d\rho)(\phi(\pi(\gamma_n(\rho))+c_N\underline{\underline{2}}_n)-\phi(\rho))\label{Conv:eq:proof:thm-nd:prelim}
\end{align}
for all $\chi\in\U_N$. As $c_N=\P(\gamma_2(\pi^N_1)=\{\{1,2\}\})$, by exchangeability of $\pi^N_1$, and by assumption \reff{Conv:eq:cond-cN},
\[1-\P(\gamma_n(\pi^N_1)=\mathbf{0}_n)
\leq \binom{n}{2}c_N\to 0\quad (N\to\infty).\]
As $\phi\in\C_n$,
\[\lim_{N\to\infty}\sup_{\rho\in\Uu_N}
\left|c_N^{-1}(\phi(\rho+c_N\underline{\underline 2}_N)-\phi(\rho))
-\langle \nabla \phi, \utwo \rangle (\rho)\right|=0,\]
this follows from the mean value theorem and uniform continuity of the derivative.
Furthermore, for every bounded measurable function $f:\R^{n^2}\to\R$,
\begin{equation}
\label{Conv:eq:dm-Ndm}
\int\nu^{N,\chi}(d\rho)f(\gamma_n(\rho))
-\int\nu^\chi(d\rho) f(\gamma_n(\rho))\leq
2\sup |f|\, n^2/N.
\end{equation}
This follows as for every semi-metric measure space $\chi=\ew{\sN,\rho,N^{-1}\sum_{i=1}^N\delta_i}\in\U_N$, we can couple $\gamma_n(\nu^{N,\chi})$ and $\gamma_n(\nu^\chi)$ by sampling $n$ times from $\sN$ uniformly with replacement and accepting this as a sample without replacement on the event that no element of $[N]$ is drawn more than once. The probability of the complementary event is bounded from above by
\[1-(1-n/N)^n\leq n^2/N.\]
Taking $f(\rho)=\langle \nabla \phi, \utwo\rangle(\rho)$, we obtain from the bound \reff{Conv:eq:dm-Ndm}
\begin{align*}
&\left|\int\nu^{N,\chi}(d\rho)c_N^{-1}(\phi(\rho+c_N\utwo_N)-\phi(\rho))
-\int\nu^\chi(d\rho)\langle \nabla \phi,\utwo\rangle (\rho)\right|\\
&\leq\int\nu^{N,\chi}(d\rho)\left|
c_N^{-1}(\phi(\rho+c_N\utwo_N)-\phi(\rho))
-\langle \nabla \phi,\utwo\rangle (\rho)\right|
+2\sup|\langle \nabla \phi,\utwo\rangle|\,n^2/N.
\end{align*}
With $f(\rho)=\phi(\pi(\rho))-\phi(\rho)$, we obtain from \reff{Conv:eq:dm-Ndm}
\begin{align*}
&\left|\int\nu^{N,\chi}(d\rho)(\phi(\pi(\gamma_n(\rho))+c_N\utwo_N)-\phi(\rho))-\int\nu^\chi(d\rho)(\phi(\pi(\gamma_n(\rho)))-\phi(\rho))\right|\\
&\leq\sup\{|\phi(\rho')-\phi(\rho)|:\rho,\rho'\in\Uu_n,\|\rho-\rho'\|\leq c_N\}\\
&\quad +4\sup|\phi|\, n^2/N.
\end{align*}
Using also Lemma~\ref{Conv:lem:conv-rates}, we now obtain from equation \reff{Conv:eq:proof:thm-nd:prelim} and the definition of $B$ in Section \ref{Conv:sec:TVFV-nd} the convergence
\[\lim_{N\to\infty}\sup_{\chi\in\U_N}
\left|c_N^{-1}(p_N-I)\Phi_N(\chi)-B\Phi(\chi)\right|=0.\]

The algebra $\Pi$ of polynomials strongly separates points in $\U$ by \cite{BK10}*{Lemma 4} and as $\Pi$ generates the topology on $\U$. Let $L$ denote the closure of $\Pi$ for the supremum norm in the space of bounded continuous functions on $\U$. Analogously to Corollaries 9.3 and 9.4 in \cite{Sampl}, the closure of $B$ generates the semigroup on $L$ of a $\U$-valued $\Xi$-Fleming-Viot process, which is strongly continuous.
The assertion now follows from Corollary 4.8.9 in \cite{EK86}, condition (h) therein is satisfied. To see that the limit on the left-hand side of equation (8.48) on p.\,234 of \cite{EK86} equals zero, we set $f=\phi$ in equation~\reff{Conv:eq:dm-Ndm}.  
\end{proof}
\begin{proof}[Proof of Theorem~\ref{Conv:thm:distr}]
Let $n\in\N$ and $N\geq n$. Let $p'_N$ be the transition kernel of the Markov chain $(\xi^N_k,k\in\N_0)$. Let $\phi\in\C_n$,
\[\Psi:\UUerg\to\R,\quad\xi\mapsto\xi\phi,\]
and
\[\Psi_N:\UU_N\to\R,\quad\nu^\chi\mapsto\nu^{N,\chi}\phi.\]
Then, as in the proof of Theorem~\ref{Conv:thm:nd},
\[p'_N\Psi_N(\xi^N_0)
=\E[\Psi_N(\xi^N_1)]
=\E[\nu^{N,\chi^N_1}\phi]
=\E[\phi(\rho^N_1)]\]
and
\begin{align*}
&c_N^{-1}(p_N-I)\Psi_N(\nu^\chi)\\
&=\P(\gamma_n(\pi^N_1)=\mathbf{0}_n)
\int\nu^{N,\chi}(d\rho)c_N^{-1}(\phi(\rho+c_N\underline{\underline 2}_N)-\phi(\rho))\\
&\quad+\sum_{\pi\in\p_n\setminus\{\mathbf{0}_n\}}c_N^{-1}\P(\gamma_n(\pi^N_1)=\pi)
\int\nu^{N,\chi}(d\rho)(\phi(\pi(\gamma_n(\rho))+c_N\underline{\underline{2}}_n)-\phi(\rho))
\end{align*}
for all $\chi\in\U_N$. As in the proof of Theorem~\ref{Conv:thm:nd}, we have
\[\lim_{N\to\infty}\sup_{\chi\in\U_N}
\left|c_N^{-1}(p_N-I)\Psi_N(\nu^\chi)-C\Psi(\nu^\chi)\right|=0.\]
The algebra $\Cc$ strongly separates points in $\UUerg$ by \cite{BK10}*{Lemma 4} and as $\Cc$ generates the weak topology on $\UUerg$. We conclude in the same way as in the proof of Theorem~\ref{Conv:thm:nd}, applying \cite{EK86}*{Corollary 4.8.9} and the analogs of Corollaries 9.3 and 9.4 in \cite{Sampl} for $\UUerg$-valued $\Xi$-Fleming-Viot processes.
\end{proof}
\begin{proof}[Proof of Proposition~\ref{Conv:prop:mod} (beginning)]
In the the first part of the proof, we assume $\Xi\in\dust$. Let $n\in\N$, $N\geq n$, $\phi\in\hat\C_n$. As in \cite{DGP12}, we associate with $\phi$ not only the marked polynomial $\Phi$ but also the marked $N$-polynomial
\[\Phi_N:\hat\U_N\to\R,\quad\chi\mapsto\nu^{N,\chi}\phi.\]
Let $\tilde p_N$ be the transition kernel of the Markov chain $(\tilde\chi^N_k,k\in\N_0)$. Similarly to the proof of Theorem~\ref{Conv:thm:nd}, we have
\[\tilde p_N\Phi_N(\tilde\chi^N_k)
=\sum_{\sigma\in\S_n}\P(\gamma_n(\sigma^N_1)=\sigma)
\int\nu^{N,\chi}(dr\,dv)\phi(\sigma(\gamma_n(r,v))+c_N(0,\uone_n)).\]
We obtain
\begin{align*}
&c_N^{-1}(\tilde p_N-I)\Phi_N(\chi)\\
&=\P(\gamma_n(\sigma^N_1)=\emptyset)
\int\nu^{N,\chi}(dr\,dv)c_N^{-1}(\phi(r,v+c_N\underline 1_N))-\phi(r,v))\\
&\quad+\sum_{\sigma\in\S_n\setminus\{\emptyset\}}c_N^{-1}\P(\gamma_n(\sigma^N_1)=\sigma)
\int\nu^{N,\chi}(dr\,dv)(\phi(\sigma(\gamma_n(r,v))+c_N(0,\underline{1}_n))-\phi(r,v))
\end{align*}
for all ${\chi}\in\hat\U_N$. By condition~\reff{Conv:eq:singl-dust} and as $\Xi\in\dust$,
\[\lim_{N\to\infty}c_N^{-1}\P(\gamma_1(\sigma^N_1)=\{\{1\}\})=\lambda_{1,\{\{1\}\}}<\infty.\]
Here we can use Lemma~\ref{Conv:lem:conv-rates-dust} or, more simply, that $b_N=\P(\gamma_1(\sigma^N_1)=\{\{1\}\})$ by exchangeability. 
Exchangeability and the assumption $\lim_{N\to\infty}c_N=0$ now 
imply
\[\P(\gamma_n(\sigma^N_1)=\emptyset)
\geq 1-n\P(\gamma_1(\sigma^N_1)=\{\{1\}\})\to 1\quad (N\to\infty).\]
Under our assumption that $\Xi\in\dust$, we conclude analogously to the proof of Theorem~\ref{Conv:thm:nd}. Here we use Lemma~\ref{Conv:lem:conv-rates-dust} from the present article, Corollaries 9.3 and 9.4 in \cite{Sampl}, and we apply Lemma 4 of \cite{BK10} to $\hat\Pi$.
\end{proof}

\subsection{Proofs of Lemmas~\ref{Conv:lem:conv-rates} and~\ref{Conv:lem:conv-rates-dust}}
\label{Conv:sec:proof-conv-rates}
For $N\in\N$ and $n\in\sN$, we define a probability kernel $\kappa^N_n$ from $\Delta^N$ to $\p_n$, and a probability kernel $K^N_n$ from $\Delta^N$ to $\S_n$.

For $x\in\Delta^N$, let $\kappa^N_N(x,\cdot)$ be the uniform distribution on those partitions in $\p_N$ whose block sizes are given by (any reordering of) the nonzero elements of the sequence $(Nx(\ell),\ell\in\N)$. Then we define $\kappa^N_n(x,\cdot)$ as the restriction $\kappa^N_n(x,\cdot)=\gamma_n(\kappa^N_N(x,\cdot))$. 

The distribution of $\kappa^N_n(x,\cdot)$ can also be described by the following urn scheme: Consider an urn that contains $Nx(\ell)$ balls of color $\ell$ for each $\ell\in\N$. Sample $n$ balls without replacement. Then the random partition of $\sn$ where $i,j\in\sn$ are in the same block if and only if the $i$-th and $j$-th ball have the same color has distribution $\kappa^N_n(x,\cdot)$.
If we modify this urn scheme such that the balls are sampled with replacement, then we obtain the distribution $\kappa^\infty_n(x,\cdot)$, as a comparison with the definition of $\kappa^\infty_n(x,\cdot)$ in Section~\ref{Conv:sec:TVFV-nd} shows.

For $x\in\Delta^N$, let $(y(1),y(2),\ldots)$ be the (possibly empty) finite subsequence of $(Nx(i),i\in\N)$ that consists of the elements that are greater or equal to $2$.
Let $K^N_N(x,\cdot)$ be the uniform distribution on those elements of $\S_N$ that consist of disjoint subsets of $\sN$ whose sizes are given by (an arbitrary reordering of) $(y(1),y(2),\ldots)$. If $(Nx(i),i\in\N)$ contains no elements that are greater or equal to $2$, then $K^N_N(x,\cdot)$ is the Dirac measure in $\emptyset\in\S_N$. We set $K^N_n(x,\cdot)=\gamma_n(K^N_N(x,\cdot))$.

In other words, consider an urn that contains $Nx(\ell)$ balls of color $\ell$ for each $\ell\in\N$. Recolor each ball whose color occurs only once with a new color $0$. Then sample $n$ balls without replacement. Consider the random element $\sigma$ of $\S_n$ where any $i,j\in\sn$ are in a common block if and only if the $i$-th and $j$-th ball have the same color that is not $0$. Then $\sigma$ has distribution $K^N_n(x,\cdot)$. When we sample with replacement, this distribution is $K^\infty_n(x,\cdot)$ instead, as a comparison with the definition of $K^\infty_n(x,\cdot)$ in Section~\ref{Conv:sec:TVFV-dust} shows.

In the following lemma and its proof, we recall and slightly extend some well-known continuity properties from e.\,g.\ Proposition 2.9 in Bertoin \cite{Bertoin}.
We endow $\N$ with the discrete topology and consider the one-point compactification $\bar\N=\N\cup\{\infty\}$.
Let $\bar\N=\N\cup\{\infty\}$ be the one-point compactification of the space $\N$, endowed with the discrete topology. We write $\Delta^\infty=\Delta$ and for $\ep>0$, we define
\[S_\ep=\{(x,N)\in\Delta_\ep\times\bar\N:x\in\Delta^N\}.\]
\begin{lem}
\label{Conv:lem:cont}
Let $n\in\N$, $\pi\in\p_n$, $\sigma\in\S_n$, and $\ep>0$. Then the maps
\[S_\ep\to[0,1],\quad (x,N)\mapsto\kappa^N_n(x,\pi)\]
and
\[S_\ep\to[0,1],\quad (x,N)\mapsto K^N_n(x,\sigma)\]
are continuous.
\end{lem}
\begin{proof}
Let $N\geq n$ and $x\in\Delta^N$. We couple the probability distributions $\kappa^\infty_n(x,\cdot)$ and $\kappa^N_n(x,\cdot)$ by starting with the urn scheme for $\kappa^\infty_n(x,\infty)$ given above and accepting the sample with replacement as a sample without replacement on the event that all sampled balls are different. The probability of this event is bounded from below by $(1-n/N)^n$, and we obtain the bound
\[|\kappa^N_n(x,\pi)-\kappa^\infty_n(x,\pi)|\leq 1-\left(1-n/N\right)^n\leq n^2/N.\]
Using the analogous coupling of the probability distributions $K^\infty_n(x,\cdot)$ and $K^N_n(x,\cdot)$, we obtain the bound
\[|K^N_n(x,\sigma)-K^\infty_n(x,\sigma)|\leq 1-\left(1-n/N\right)^n\leq n^2/N.\]
Furthermore, for $x,y\in\Delta$, we can couple the probability distributions $\kappa^\infty_n(x,\cdot)$ and $\kappa^\infty_n(y,\cdot)$ by using the same uniform random variables in Kingman's correspondence which we recalled in Section~\ref{Conv:sec:TVFV-nd}. This yields
\[|\kappa^\infty_n(x,\pi)-\kappa^\infty_n(y,\pi)|\leq n|x-y|_1.\]
Similarly, using the definition of $K^\infty_n$ from Section~\ref{Conv:sec:TVFV-dust}, we obtain
\[|K^\infty_n(x,\sigma)-K^\infty_n(y,\sigma)|\leq n|x-y|_1.\]

W.\,l.\,o.\,g., let $((x_k,N_k),k\in\N)$ be a sequence in $S_\ep$ that converges to some $(x,\infty)\in S_\ep$. Then,
\begin{align*}
&|K^{N_k}_n(x_k,\sigma)-K^\infty_n(x,\sigma)|\\
&\leq|K^{N_k}_n(x_k,\sigma)-K^\infty_n(x_k,\sigma)|
+|K^\infty_n(x_k,\sigma)-K^\infty_n(x,\sigma)|\\
&\leq n^2/N_k+n|x_k-x|_1,
\end{align*}
and the right-hand side converges to zero as $k$ tends to infinity. The argument for 
$|\kappa^{N_k}_n(x_k,\pi)-\kappa^\infty_n(x,\pi)|$ is the same.
\end{proof}
\begin{proof}[Proof of Lemma~\ref{Conv:lem:conv-rates}]
Let $n\in\N$. By construction, we have
\[\P(\gamma_n(\pi^N_1)=\pi)=\int\kappa^N_n(x,\pi)\Xi^N(dx)\]
for all $\pi\in\p_n$ and $N\geq n$.
For arbitrarily small $\ep>0$, assumption~\reff{Conv:eq:cond-conv} implies the weak convergence
\[c_N^{-1}\Xi^N(dx)\delta_N(dN')\wto\left|x\right|_2^{-2}\Xi(dx)\delta_\infty(dN')\quad\text{on }\Delta_\ep\times\bar\N\quad(N\to\infty)\]
which also holds on $S_\ep$ as none of these measures have mass on the complement of $S_\ep$ in $\Delta_\ep\times\bar\N$. Using Lemma~\ref{Conv:lem:cont}, we obtain for each $\pi\in\p_n$
\begin{equation}
\label{Conv:eq:conv-nd}
\begin{aligned}
&\lim_{N\to\infty}c_N^{-1}\int_{\Delta_\ep}\kappa^N_n(x,\pi)\Xi^N(dx)
=\lim_{N\to\infty}\int_{S_\ep}\kappa^{N'}_n(x,\pi)c_N^{-1}\Xi^N(dx)\delta_N(dN')\\
&=\int_{S_\ep}\kappa^{N'}_n(x,\pi)\left|x\right|_2^{-2}\Xi(dx)\delta_\infty(dN')
=\int_{\Delta_\ep}\kappa^\infty_n(x,\pi)\left|x\right|_2^{-2}\Xi(dx).
\end{aligned}
\end{equation}

For every sufficiently small $\ep>0$, for $N\geq n$, and $x\in\Delta^N\setminus\Delta_\ep$, the urn scheme for the probability distribution $\kappa^N_n(x,\cdot)$ yields that
\begin{align}
&\kappa^N_2(x,\{\{1,2\}\})(1-n\ep)^n\notag\\
&\leq
\kappa^N_2(x,\{\{1,2\}\})\frac{N-\ep N}{N-2}\cdots\frac{N-(n-2)\ep N}{N-(n-1)}\notag\\
&\leq\kappa^N_n(x,\{\{1,2\},\{3\},\ldots,\{n\}\})\notag\\
&\leq\kappa^N_2(x,\{\{1,2\}\}).
\label{Conv:eq:conv-kappa2}
\end{align}
Indeed, we can draw a partition of $\{1,\ldots,n\}$ according to the distribution $\kappa^N_n(x,\cdot)$ by drawing $n$ times without replacement from $N$ individuals that are subdivided into families of sizes $(Nx(\ell),\ell\in\N)$, and letting $i$ and $j$ be in the same block if and only if the $i$-th and $j$-th drawn individual are in the same family. The second inequality follows as the family sizes are at most $\ep N$ by our assumption on $x$.

For $x\in\Delta\setminus\Delta_\ep$ and every partition $\pi\in\p_n$ that contains one doubleton and apart from that only singletons, we have
\begin{equation}
\label{Conv:eq:conv-kappa3}
\kappa^N_n(x,\pi)=\kappa^N_n(x,\{\{1,2\},\{3\},\ldots,\{n\}\})
\end{equation}
by exchangeability.
Moreover,
\begin{equation}
\label{Conv:eq:conv-c_N}
c_N=\int\kappa^N_2(x,\{\{1,2\}\})\Xi^N(dx)
\end{equation}
by exchangeability.

Consequently, for arbitrarily small $\ep>0$ as in condition~\reff{Conv:eq:cond-conv}, and any partition $\pi\in\p_n$ that contains one doubleton and apart from that only singletons,
\begin{align*}
&\limsup_{N\to\infty}c_N^{-1}\int_{\Delta\setminus\Delta_\ep}\kappa^N_n(x,\pi)\Xi^N(dx)
\leq\limsup_{N\to\infty}c_N^{-1}\int_{\Delta\setminus\Delta_\ep}\kappa^N_2(x,\{\{1,2\}\})\Xi^N(dx)\\
&=1-\liminf_{N\to\infty}c_N^{-1}\int_{\Delta_\ep}\kappa^N_2(x,\{\{1,2\}\})\Xi^N(dx)
=1-\Xi(\Delta_\ep),
\end{align*}
where we use \reff{Conv:eq:conv-kappa2} and \reff{Conv:eq:conv-kappa3} for the first, equation \reff{Conv:eq:conv-c_N} for the second, and the convergence~\reff{Conv:eq:conv-nd} and $\kappa^\infty_2(x,\{\{1,2\}\})=\left|x\right|_2^2$ for the third step.
Analogously, we obtain for the same partitions $\pi$ that
\[\liminf_{N\to\infty}c_N^{-1}\int_{\Delta\setminus\Delta_\ep}\kappa^N_n(x,\pi)\Xi^N(dx)
\geq (1-n\ep)^{n}(1-\Xi(\Delta_\ep)).\]

For every partition $\pi\in\p_n$ that contains a block of size greater than $2$, $N\geq n$, and $x\in\Delta^N\setminus\Delta_\ep$ we have
\[\kappa^N_n(x,\pi)\leq\kappa^N_3(x,\{\{1,2,3\}\})\leq\kappa^N_2(x,\{\{1,2\}\})\frac{N\ep-2}{N-2}\]
which can be seen from the urn scheme for $\kappa^N_n(x,\cdot)$ and the exchangeability therein. We obtain for these partitions $\pi$ that
\[\limsup_{N\to\infty}c_N^{-1}\int_{\Delta\setminus\Delta_\ep}\kappa^N_n(x,\pi)\Xi^N(dx)
\leq \ep(1-\Xi(\Delta_\ep)).\]

Similarly, for every partition $\pi\in\p_n$ that contains more than one non-singleton block, $N\geq n$, and $x\in\Delta^N\setminus\Delta_\ep$, we have
\[\kappa^N_n(x,\pi)\leq\kappa^N_4(x,\{\{1,2\},\{3,4\}\})\leq\kappa^N_2(x,\{\{1,2\}\})\frac{N\ep-1}{N-3},\]
hence
\[\limsup_{N\to\infty}c_N^{-1}\int_{\Delta\setminus\Delta_\ep}\kappa^N_n(x,\pi)\Xi^N(dx)
\leq \ep(1-\Xi(\Delta_\ep)).\]

For $\pi\in\p_n\setminus\{\mathbf{0}_n\}$, we write
\[c_N^{-1}\P(\gamma_n(\pi^N_1)=\pi)
=c_N^{-1}\int_{\Delta_\ep}\kappa^N_n(x,\pi)\Xi^N(dx)
+c_N^{-1}\int_{\Delta\setminus\Delta_\ep}\kappa^N_n(x,\pi)\Xi^N(dx).\]
Now we let first $N\to\infty$. For the first integral, we use the convergence~\reff{Conv:eq:conv-nd}. For the second integral, we use the bounds for the limes superior and if necessary also for the limes inferior. Then we let $\ep$ tend to zero along a sequence such that for each element of this sequence, the weak convergence in condition~\reff{Conv:eq:cond-conv} holds. Thus we obtain convergence to the rates $\lambda_\pi$ as asserted in Lemma~\ref{Conv:lem:conv-rates}.
\end{proof}
\begin{proof}[Proof of Lemma~\ref{Conv:lem:conv-rates-dust}]
The proof is analogous to Lemma \ref{Conv:lem:conv-rates}. By construction, we have
\[\P(\gamma_n(\sigma^N_1)=\sigma)=\int K^N_n(x,\sigma)\Xi^N(dx)\]
for all $\sigma\in\S_n$ and $N\geq n$. Analogously to the convergence \reff{Conv:eq:conv-nd}, we obtain from Lemma \ref{Conv:lem:cont} and assumption~\reff{Conv:eq:cond-conv} that
\begin{equation}
\label{Conv:eq:conv-dust}
\lim_{N\to\infty}c_N^{-1}\int_{\Delta_\ep}K^N_n(x,\sigma)\Xi^N(dx)
=\int_{\Delta_\ep}K^\infty_n(x,\sigma)\left|x\right|_2^{-2}\Xi(dx)
\end{equation}
for all $\sigma\in\S_n$ and arbitrarily small $\ep>0$.

For every $\sigma\in\S_n\setminus\{\emptyset\}$, $N\geq n$, and $x\in\Delta^N$, it holds
\[K^N_n(x,\sigma)\leq K^N_1(x,\{\{1\}\}),\]
by exchangeability. Furthermore,
\[b_N=\int K^N_1(x,\{\{1\}\})\Xi^N(dx).\]
Now we obtain for these $\sigma$ and arbitrarily small $\ep>0$
\begin{align}
&\limsup_{N\to\infty}c_N^{-1}\int_{\Delta\setminus\Delta_\ep}K^N_n(x,\sigma)\Xi^N(dx)
\leq\limsup_{N\to\infty}c_N^{-1}\int_{\Delta\setminus\Delta_\ep}K^N_1(x,\{\{1\}\})\Xi^N(dx)\notag\\
&=\lim_{N\to\infty}b_N/c_N-\lim_{N\to\infty}c_N^{-1}\int_{\Delta_\ep}K^N_1(x,\{\{1\}\})\Xi^N(dx)
=\int_{\Delta\setminus\Delta_\ep}\left|x\right|_1\left|x\right|_2^{-2}\Xi(dx),
\label{Conv:eq:K-bound-dust}
\end{align}
where we use assumption~\reff{Conv:eq:singl-dust}, the convergence~\reff{Conv:eq:conv-dust}, and $K^\infty_1(x,\{\{1\}\})=\left|x\right|_1$ in the last equality.
We conclude by combining the convergence~\reff{Conv:eq:conv-dust} and the bound~\reff{Conv:eq:K-bound-dust}. We let $\ep$ tend to zero and use dominated convergence as the integrands on the right-hand sides of \reff{Conv:eq:conv-dust} and \reff{Conv:eq:K-bound-dust} are bounded by $|x|_1|x|_2^{-2}$, and as
\[\int_\Delta\left|x\right|_1\left|x\right|_2^{-2}\Xi(dx)<\infty.\]
This yields the assertion.
\end{proof}

\section{Convergence of marked metric measure spaces in the dust-free case}
\label{Conv:sec:mmm-nd}
In this section, we prove Theorem~\ref{Conv:thm:dust} and Proposition~\ref{Conv:prop:mod} for $\Xi\in\nd$.
We use the isometric embedding $\beta_0:\U_N\to\hat\U_N$, $\ew{X,\rho,\mu}\mapsto\ew{X,\rho,\mu\otimes\delta_0}$ which maps a metric measure space to the associated marked metric measure space that supports only the zero mark.
We compare the process $(\beta_0(\chi^N_{\fl{c_N^{-1}t}}),t\in\R_+)$ to the processes $(\hat\chi^N_{\fl{c_N^{-1}t}},t\in\R_+)$ and $(\tilde\chi^N_{\fl{c_N^{-1}t}},t\in\R_+)$.

Recall from Section~\ref{Conv:sec:dec} the map $\beta:\U_N\to\hat\U_N$ which decomposes a tree at the external branches. Also recall from Section \ref{Conv:sec:dist-dec} the functions $\alpha$ and $\Upsilon$. We denote also by $\alpha$ the function\label{Conv:not:alphamm} from $\hat\U_N$ to $\U_N$ that maps $\chi$ to $\psi_N(\alpha(r,v))$, where $(r,v)$ is any element of $\hat\Uu_N$ with $\hat\psi_N(r,v)=\chi$.
The function $\alpha:\hat\U_N\to\U_N$ retrieves a metric measure space from a marked metric measure space by adding the marks to the metric distances.
For $\ell\geq 2$, we define the map
\[\Upsilon^\ell_1:\Uu\to\R_+,\quad \Upsilon^\ell_1=\gamma_1\circ\Upsilon\circ\gamma_\ell.\]
In the subtree spanned by the first $\ell$ leaves of the tree associated with some $\rho\in\Uu$, the length of the external branch that ends in the first leaf is given by $\Upsilon^\ell_1(\rho)$. We also define the restriction $\varpi:\hat\Uu\to\R_+$, $(r,v)\mapsto v(1)$.

We recall that the Prohorov metric $\dP$ on the space of probability measures on the Borel sigma algebra of a metric space $(S,d)$ is given by
\begin{equation}
\label{Conv:eq:def-dP}
\dP(\mu,\mu')=\inf\{\ep>0:\mu'(F)\leq\mu(F^\ep)+\ep\text{ for all closed }F\subset S\},
\end{equation}
where $F^\ep=\{x\in S:d(x,F)<\ep\}$.
If $(S,d)$ is separable, then we also have the coupling characterization
\begin{equation}
\label{Conv:eq:dP-coupl}
\dP(\mu,\mu')=\inf_\nu\inf\{\ep>0:\nu\{(x,y)\in S^2:d(x,y)>\ep\}<\ep\}
\end{equation}
where the first infimum is over all couplings $\nu$ of $\mu$ and $\mu'$, see e.\,g.\ \cite{EK86}*{Theorem 3.1.2}.
\begin{lem}
\label{Conv:lem:bound-MGP}
Let $\chi\in\hat\U_N$. Then,
\[\dMGP({\chi},\beta_0\circ\alpha(\chi))\leq\dP(\varpi(\nu^{\chi}),\delta_0).\]
\end{lem}
\begin{cor}
\label{Conv:cor:bound-MGP}
For every $\chi\in\U_N$,
\[\dMGP(\beta(\chi),\beta_0(\chi))\leq\dP(\varpi(\nu^{\beta(\chi)}),\delta_0).\]
\end{cor}
\begin{proof}
This is immediate from Lemma~\ref{Conv:lem:bound-MGP} as $\alpha\circ\beta$ is the identity on $\U_N$.
\end{proof}
We prove Lemma~\ref{Conv:lem:bound-MGP} using a characterization of the marked Gromov-Prohorov metric $\dMGP$ that we will apply also in Section~\ref{Conv:sec:mmm-dust}.
The distortion $\dis{\mathfrak{R}}$ of a relation $\mathfrak{R}\subset X\times X'$ between two metric spaces $(X,r)$ and $(X',r')$ is defined by
\[\dis{\mathfrak{R}}=\sup\{\left|r(x,y)-r'(x',y')\right|:(x,x'),(y,y')\in\mathfrak{R}\}.\]
\begin{prop}
\label{Conv:prop:MGP-rel}
Let $(X,r,m)$ and $(X',r',m')$ be marked metric measure spaces. Then\\
$\dMGP(\ew{X,r,m},\ew{X',r',m'})$ is the infimum of all $c>0$ such that there exist a relation $\mathfrak{R}\subset X\times X'$ and a coupling $\nu$ of $m$ and $m'$ with $\tfrac{1}{2}\dis\mathfrak{R}\leq c$ and $\nu(\mathfrak{\tilde R})\geq 1-c$, where $\mathfrak{\hat R}\subset(X\times\R_+)\times(X'\times\R_+)$ is defined by
\[\mathfrak{\hat R}=\{((x,v),(x',v')):(x,x')\in\mathfrak{R},|v-v'|\leq c\}.\]
\end{prop}
\begin{proof}
This can be seen as an adaptation of Proposition 6 in \cite{Miermont09}. Here we sketch the proof of the upper bound for $\dMGP(\ew{X,r,m},\ew{X',r',m'})$. Let $c>0$ and assume $\mathfrak{R}$, $\mathfrak{\hat{R}}$, and $\nu$ with $\tfrac{1}{2}\dis\mathfrak{R}\leq c$ and $\nu(\mathfrak{\hat R})\geq 1- c$ are given as in the proposition. A metric $d^Z$ on the disjoint union $Z=X\sqcup X'$ can be defined by
$d^Z(x,y)=r(x,y)$ for $(x,y)\in X$,
$d^Z(x,y)=r'(x,y)$ for $(x,y)\in X'$, and
\[d^Z(x,x')=\inf\{r(x,y)+c+r'(y',x'):(y,y')\in\mathfrak{R}\},\]
for $x\in X,x'\in X'$, as in Remark 5.5 of \cite{GPW09}.
We endow $Z\times\R_+$ with the product metric
$d^{Z\times\R_+}((z,v),(z',v'))=d^Z(z,z')\vee\left|v-v'\right|$.
Let $\varphi:X\to Z$ and $\varphi':X'\to Z$ be the canonical embeddings. Moreover, let $\hat\varphi(x,v)=(\varphi(x),v)$ and $\hat\varphi'(x',v)=(\varphi'(x'),v)$ for $x\in X$, $x'\in X'$, and $v\in\R_+$. Then the coupling $\nu$ induces a coupling $\hat\nu$ of $\hat\varphi(m)$ and $\hat\varphi'(m')$ on $Z\times\R_+$ with
\[\hat\nu\{((z,v),(z',v')):d^Z(z,z')\vee|v-v'|\leq c\}\geq 1- c.\]
The coupling characterization \reff{Conv:eq:dP-coupl} of the Prohorov metric implies
$\dP(\hat\varphi(m),\hat\varphi'(m'))\leq c$. The definition of the marked Gromov-Prohorov metric, see \cite{DGP12}*{Definition 3.1}, implies
\[\dMGP(\ew{X,r,m},\ew{X',r',m'})\leq c.\]
\end{proof}

\begin{proof}[Proof of Lemma~\ref{Conv:lem:bound-MGP}]
Let $(r,v)$ be any element of $\hat\Uu_N$ with $\hat\psi_N(r,v)={\chi}$, and let $\rho=\alpha(r,v)$. Then,
\[\ew{\sN,r,N^{-1}\sum_{i=1}^N\delta_{(i,v(i))}}={\chi}\]
and
\[\ew{\sN,\rho,N^{-1}\sum_{i=1}^N\delta_{(i,0)}}=\beta_0\circ\alpha({\chi}).\]

Let $c>\dP(\varpi(\nu^{\chi}),\delta_0)$. We conclude by Proposition~\ref{Conv:prop:MGP-rel} which also holds for marked semi-metric measure spaces. To this aim, we define the relation
\[\fR=\{(i,i)\in\sN\times\sN:v(i)\leq c\}\]
between the semi-metric spaces $(\sN,r)$ and $(\sN,\rho)$. As $|r(i,j)-\rho(i,j)|\leq v(i)+v(j)$ by definition of the map $\alpha:\hat\Uu_N\to\Uu_N$, we can bound the distortion
by
\[\dis\fR=\max\{|r(i,j)-\rho(i,j)|:i,j\in\sN,v(i),v(j)\leq c\}\leq 2c.\]
We set
\[\hat\fR=\{((i,v(i)),(i,0)):i\in\sN,v(i)\leq c\}\subset(\sN\times\R_+)\times(\sN\times\R_+).\]
A coupling $\nu$ of the probability measures $N^{-1}\sum_{i=1}^N\delta_{(i,v(i))}$ and $N^{-1}\sum_{i=1}^N\delta_{(i,0)}$ is given by
\[\nu=N^{-1}\sum_{i=1}^N\delta_{((i,v(i)),(i,0))}.\]
Finally,
\[\nu(\hat\fR)=N^{-1}\sum_{i=1}^N\I{v(i)\leq c}=\varpi(\nu^{\chi})[0,c]\geq \delta_0\{0\}-c=1-c,\]
where the inequality follows from the choice of $c$ and the usual definition \reff{Conv:eq:def-dP} of the Prohorov metric.
\end{proof}
\begin{lem}
\label{Conv:lem:ext}
Let $\Xi\in\nd$, and let $(\chi_t,t\in\R_+)$ be a $\U$-valued $\Xi$-Fleming-Viot process with càdlàg paths. Then,
\[\lim_{\ell\to\infty}\sup_{t\in[0,T]}\dP(\Upsilon^\ell_1(\nu^{\chi_t}),\delta_0)=0\quad\text{a.\,s.}\]
for all $T\in\R_+$.
\end{lem}
\begin{proof}
Let $\ep>0$. For $\ell\geq 2$, we define the random time
\[\vartheta_{\ep,\ell}=\inf\{t\in\R_+:\dP(\Upsilon^\ell_1(\nu^{\chi_t}),\delta_0)>\ep\}.\]
For all $\rho\in\Uu$, the map $\ell\mapsto \Upsilon^\ell_1(\rho)$ is non-increasing. We set
\[\vartheta_\ep=\sup_{\ell\in\N}\vartheta_{\ep,\ell}
=\lim_{\ell\to\infty}\vartheta_{\ep,\ell}.\]
Let $t\in(0,\infty)$ be arbitrary. On an event of probability $1$, let $(X,\rho,\mu)$
be a representative of $\chi_{\vartheta_\ep\wedge t}$, and let $(x_i,i\in\N)$ be a $\mu$-iid sequence in $X$. Then,
\[\inf_{j\in\N\setminus\{1\}}\rho(x_1,x_j)=0\quad\text{a.\,s.}\]
as an iid sequence with respect to a probability measure on the Borel sigma algebra on a separable metric space has a.\,s.\ no isolated elements.
Consequently,
\[\lim_{\ell\to\infty}\Upsilon^\ell_1((\rho(x_i,x_j))_{i,j\in\N})=0\quad\text{a.\,s.}\]
As the random matrix $(\rho(x_i,x_j))_{i,j\in\N}$ has conditional distribution $\nu^{\chi_{\vartheta_\ep\wedge t}}$ given $\chi_{\vartheta_\ep\wedge t}$, it follows
\[\lim_{\ell\to\infty}\dP(\Upsilon^\ell_1(\nu^{\chi_{\vartheta_\ep\wedge t}}),\delta_0)=0\quad\text{a.\,s.}\]
Analogously, it can be shown that
\[\lim_{\ell\to\infty}\dP(\Upsilon^\ell_1(\nu^{\chi_{(\vartheta_\ep\wedge t)-}}),\delta_0)=0\quad\text{a.\,s.}\]

As $(\chi_s,s\in\R_+)$ has càdlàg paths and as the maps $\chi\mapsto\nu^\chi$ and $\Upsilon^\ell_1$ are continuous, it follows that on an event of probability $1$, there exists $\ell\in\N$ such that
\begin{equation*}
\dP(\Upsilon^\ell_1(\nu^{\chi_s}),\delta_0)<\ep/2
\end{equation*}
for all $s$ in a neighborhood of $\vartheta_\ep\wedge t$. By monotonicity, it also holds
\[\dP(\Upsilon^{\ell'}_1(\nu^{\chi_s}),\delta_0)<\ep/2\]
for all $\ell'\geq\ell$ and $s$ in the same neighborhood, on the same event of probability $1$. This implies $\vartheta_{\ell,\ep}>t$ for $\ell$ sufficiently large a.\,s., hence $\{\vartheta_\ep<t\}$ is a null event. As $t$ was arbitrary, it follows $\vartheta_\ep=\infty$ a.\,s.\ which yields the assertion.
\end{proof}
\begin{proof}[Proof of Theorem~\ref{Conv:thm:dust} (beginning)]
First we assume $\Xi\in\nd$. Let $T\in\R_+$. By the assumption that $\hat\chi_0$ supports only the zero mark, there exists $\chi_0\in\U$ such that $\hat\chi^N_0$ converges to $\beta_0(\chi_0)$ in the marked Gromov-weak topology as $N$ tends to infinity. Hence, $\nu^{\hat\chi^N_0}$ converges weakly to $\nu^{\chi_0}\otimes\delta_0$. Recall the chain $(\chi^N_k,k\in\N_0)$ from Section~\ref{Conv:sec:inv-mm}. As $\nu^{\chi^N_0}=\alpha(\nu^{\hat\chi^N_0})$ converges weakly to $\nu^{\chi_0}$, also $\chi^N_0$ converges to $\chi_0$ in the Gromov-weak topology.
Hence, Theorem~\ref{Conv:thm:nd} is applicable and the processes $(\chi^N_{\fl{c_N^{-1}t}},t\in\R_+)$ converge in distribution to a $\U$-valued $\Xi$-Fleming-Viot process $(\chi_t,t\in\R_+)$ with initial state $\chi_0$ in the space of càdlàg paths in $(\U,\dGP)$, endowed with the Skorohod metric. For every $\ell\geq 2$, by continuity of the maps $\chi\mapsto\nu^\chi$ and $\Upsilon^\ell_1$, also $(\Upsilon^\ell_1(\nu^{\chi^N_{\fl{c_N^{-1}t}}}),t\in\R_+)$ converges in distribution to $(\Upsilon^\ell_1(\nu^{\chi_t}),t\in\R_+)$ in the space of càdlàg paths in $(\Mm_1(\R_+),\dP)$, endowed with the Skorohod metric, where $\Mm_1(\R_+)$ denotes the space of probability measures on $\R_+$.

For every $(r,v)\in\hat\Uu$ and $\rho=\alpha(r,v)$, it holds $\rho(1,j)\geq v(1)$ for all $j\geq 2$ by definition of the map $\alpha:\hat\Uu_N\to\Uu_N$, hence $v(1)\leq 2\Upsilon^\ell_1(\alpha(r,v))$ for all $\ell\geq 2$. This implies
\begin{align}
&\sup_{t\in[0,T]}\dP(\varpi(\nu^{\hat\chi^N_{\fl{c_N^{-1}t}}}),\delta_0)\label{Conv:eq:nd-chi-bound}\\
&\leq\sup_{t\in[0,T]}\dP(2\Upsilon^\ell_1(\alpha(\nu^{\hat\chi^N_{\fl{c_N^{-1}t}}})),\delta_0)\notag\\
&\leq2\sup_{t\in[0,T]}\dP(\Upsilon^\ell_1(\nu^{\chi^N_{\fl{c_N^{-1}t}}}),\delta_0).\notag
\end{align}
The expression in the last line converges in distribution to
\begin{equation}
\label{Conv:eq:thm-dust-nd-expr}
2 \sup_{t\in[0,T]}\dP(\Upsilon^\ell_1(\nu^{\chi_t}),\delta_0)
\end{equation}
as $N$ tends to infinity. This follows from the discussion in the beginning of this proof, as the maps $\Upsilon^\ell_1$ and $\dP(\cdot,\delta_0)$ are continuous, and as the process $(\Upsilon^\ell_1(\nu^{\chi_t}),t\in\R_+)$ has a.\,s.\ no discontinuity at the fixed time $T$. Finally, we let $\ell$ tend to infinity. By Lemma~\ref{Conv:lem:ext}, expression~\reff{Conv:eq:thm-dust-nd-expr} then converges to zero a.\,s. Consequently, also the left-hand side of~\reff{Conv:eq:nd-chi-bound} converges to zero in probability as $N$ tends to infinity. 

As $\hat\chi^N_k=\beta(\chi^N_k)$ for all $k\in\N_0$, Corollary~\ref{Conv:cor:bound-MGP} implies that
\[\sup_{t\in[0,T]}\dMGP(\hat\chi^N_{\fl{c_N^{-1}t}},\beta_0(\chi^N_{\fl{c_N^{-1}t}}))\]
converges to zero in probability.
The processes $(\beta_0(\chi^N_{\fl{c_N^{-1}t}}), t\in\R_+)$ converge in distribution to $(\beta_0(\chi_t),t\in\R_+)$ by another application of Theorem~\ref{Conv:thm:nd}.
The assertion for $\Xi\in\nd$ now follows from Slutzky's theorem.
\end{proof}
To prove Proposition~\ref{Conv:prop:mod} in case $\Xi\in\nd$, we will use the following coupling of $(\tilde\chi^N_k,k\in\N_0)$ and $(\chi^N_k,k\in\N_0)$.
\begin{rem}
\label{Conv:rem:coupl-mm-mod}
If $\chi^N_0=\alpha(\tilde\chi^N_0)$, then we can define $(r^N_0,v^N_0)$ and $\rho^N_0$ in Section \ref{Conv:sec:results} such that
\[\rho^N_0=\alpha(r^N_0,v^N_0).\]
We can then also assume that the processes $((r^N_k,v^N_k),k\in\N_0)$, $(\tilde\chi^N_k,k\in\N_0)$, and $(\rho^N_k,k\in\N_0)$ are defined as in Section \ref{Conv:sec:mod}. Then we obtain
$\rho^N_k=\alpha(r^N_k,v^N_k)$ for all $k\in\N$. When we define $\chi^N_k=\psi_N(\rho^N_k)$ and
$\hat\chi^N_k=\hat\psi_N\circ\beta(\rho^N_k)$ as in Section \ref{Conv:sec:results}, then
we also have $\chi^N_k=\alpha(\tilde\chi^N_k)$ and $\hat\chi^N_k=\beta\circ\alpha(\tilde\chi^N_k)$ for all $k\in\N_0$.
\end{rem}
\begin{proof}[Proof of Proposition~\ref{Conv:prop:mod} (end)]
In case $\Xi\in\nd$, the proof is almost identical with the proof of Theorem~\ref{Conv:thm:dust}. We replace $\hat\chi^N_\cdot$ with $\tilde\chi^N_\cdot$, and we set $\chi^N_0=\alpha(\tilde\chi^N_0)$. Then we use the coupling from Remark~\ref{Conv:rem:coupl-mm-mod} and apply Theorem~\ref{Conv:thm:nd}.
\end{proof}

\section{Convergence of marked metric measure spaces in the case with dust}
\label{Conv:sec:mmm-dust}
In this section, we complete the proof of Theorem~\ref{Conv:thm:dust}: In the case with dust, we compare the finite dimensional distributions of the processes $(\hat\chi^N_{\fl{c_N^{-1}t}},t\in\R_+)$ and $(\tilde\chi^N_{\fl{c_N^{-1}t}},t\in\R_+)$, and we show relative compactness of the sequence of processes $((\hat\chi^N_{\fl{c_N^{-1}t}},t\in\R_+),N\geq 2)$.

A metric $\tilde d$ on $\hat\Uu_N$ is defined by
\[\tilde d((r,v),(r',v'))=\max_{i,j\in\sN}|r(i,j)-r'(i,j)|\vee\max_{i\in\sN}|v(i)-v'(i)|.\]
W.\,l.\,o.\,g., we endow $\hat\Uu_N$ with the metric
\[d((r,v),(r',v'))=\max_{n\in\sN}(\tilde d(\gamma_n(r,v),\gamma_n(r',v'))\wedge (2^{-n}))\]
and $\hat\Uu$ with the metric
\[d((r,v),(r',v'))=\sup_{n\in\N}(\tilde d(\gamma_n(r,v),\gamma_n(r',v'))\wedge  (2^{-n})).\]
These metrics induce the topologies on $\hat\Uu_N$ and $\Uu$ from Section \ref{Conv:sec:dist-dec}. We have
\begin{equation}
\label{Conv:eq:d-U-restr}
d(\gamma_n(r,v),\gamma_n(r',v'))
\leq d((r,v),(r',v'))
\leq d(\gamma_n(r,v),\gamma_n(r',v'))+2^{-n}
\end{equation}
for all $(r,v),(r',v')\in\hat\Uu$.

We use definitions in particular from Section \ref{Conv:sec:dist-dec} and we need the following lemmas.
\begin{lem}
\label{Conv:lem:rv-v}
Let $n\in\N$ and $(r,v),(r',v')\in\hat\Uu_n$ with $\alpha(r,v)=\alpha(r',v')$. Then,
\[d((r,v),(r',v'))\leq 2\max_{i\in\sn}\left|v(i)-v'(i)\right|.\]
\end{lem}
\begin{proof}
Let $\rho=\alpha(r,v)=\alpha(r',v')$. By definition of $\alpha$,
\[|r(i,j)-r'(i,j)|=|\rho(i,j)-v(i)-v(j)-\rho(i,j)+v'(i)+v'(j)|
\leq|v(i)-v'(i)|+|v(j)-v'(j)|\]
for all distinct $i,j\in\sn$. It follows
\[d((r,v),(r',v'))\leq\tilde d((r,v),(r',v'))\leq
2\max_{i\in\sn}\left|v(i)-v'(i)\right|.\]
\end{proof}

In the next lemma, we consider the chain $(\tilde\chi^N_k,k\in\N_0)$ from Section \ref{Conv:sec:mod}. We show that if the initial state corresponds to the decomposition at the external branches, then in all generations, the mark of a sampled individual is not larger than the length of the corresponding external branch.
\begin{lem}
\label{Conv:lem:ext-v}
Let the $\hat\U_N$-valued chain $(\tilde\chi^N_k,k\in\N_0)$ be defined from the $\hat\Uu_N$-valued chain $((r^N_k,v^N_k),k\in\N_0)$ as in Section \ref{Conv:sec:mod}, and assume that $(r^N_0,v^N_0)=\beta\circ\alpha(r^N_0,v^N_0)$.
Then,
\[v\leq\Upsilon\circ\alpha(r,v)\quad\text{(component-wise)}\]
for $\nu^{N,\tilde\chi^N_k}$-a.\,a.\ $(r,v)\in\hat\Uu_N$ and all $k\in\N_0$ a.\,s.
\end{lem}
\begin{proof}
Let $(\rho^N_k,k\in\N_0)=(\alpha(r^N_k,v^N_k),k\in\N_0)$ as in Section \ref{Conv:sec:mod}.
By definition of $\Upsilon$,
\begin{equation}
\label{Conv:eq:lem-mmm-v-barv}
\Upsilon(\rho^N_k)(i)=\tfrac{1}{2}\min_{j\in\sN\setminus\{i\}}\rho^N_k(i,j)
\end{equation}
for all $i\in[N]$ and $k\in\N_0$.
We show $v^N_k(i)\leq \Upsilon(\rho^N_k)(i)$. The assertion follows as for $\nu^{N,\tilde\chi^N_k}$-a.\,a.\ $(r,v)\in\hat\Uu_N$, there exists a.\,s.\ a bijection $p$ on $[N]$ with $v(i)=v^N_k(p(i))$ for all $i\in[N]$ by definition of the $N$-marked distance matrix distribution.

Let $i\in[N]$, $k\in\N_0$, and $\bar v(i)=\Upsilon(\rho^N_k(i))$.
To show $v^N_k(i)\leq\bar v(i)$, we consider the cases $v^N_k(i)\leq c_Nk$ and $v^N_k(i)\geq c_Nk$ separately.

In the case $v^N_k(i)\leq c_Nk$, there are, by definition of $v^N_k$, no reproduction events from generations $k-c_N^{-1}v^N_k(i)+1$ to $k$ due to which the ancestral lineage of the individual $i$ of generation $k$ can merge with the ancestral lineage of a different individual. Hence, by definition of $\rho^N_k$ in Section \ref{Conv:sec:Cannings},
\[\min_{j\in\sN\setminus\{i\}}\rho^N_k(i,j)\geq 2v^N_k(i).\]
Equation~\reff{Conv:eq:lem-mmm-v-barv} yields $v^N_k(i)\leq\bar v(i)$ in case $v^N_k(i)\leq c_Nk$.

The statements in the remainder of this proof hold in the case $v^N_k(i)\geq c_Nk$.
There are no reproduction events in which the ancestral lineage of the individual $i$ of generation $k$ can merge with the ancestral lineage of a different individual. Hence, $A_0(k,i)\neq A_0(k,j)$ for all $j\in\sN\setminus\{i\}$. By definition of $\rho^N_k$,
\begin{equation}
\label{Conv:eq:lem-mmm-rhoN}
\rho^N_k(i,j)=2c_Nk+\rho^N_0(A_0(k,i),A_0(k,j))
\end{equation}
for all $j\in\sN\setminus\{i\}$. By definition of $v^N_k$,
\[v^N_k(i)=c_Nk+v^N_0(A_0(k,i)).\]
Also, $v^N_0=\Upsilon(\rho^N_0)$ by our assumption, hence
\begin{equation}
\label{Conv:eq:lem-mmm-vN}
v^N_0(A_0(k,i))\leq\tfrac{1}{2}\rho^N_0(A_0(k,i),\ell)
\end{equation}
for all $\ell\in\sN\setminus\{A_0(k,i)\}$. Using equations~\reff{Conv:eq:lem-mmm-v-barv}, \reff{Conv:eq:lem-mmm-rhoN}, and \reff{Conv:eq:lem-mmm-vN}, we obtain
\[\bar v(i)=\tfrac{1}{2}\min_{j\in\sN\setminus\{p(i)\}}\rho^N_k(i,j)
\geq c_Nk+v^N_0(A_0(k,i))=v^N_k(i).\]
\end{proof}

For $\ell\geq n\geq 2$, we introduce the map
\[\Upsilon^\ell_n:\Uu\cup\bigcup_{N\geq \ell}\Uu_N\to\R_+^n,\quad\Upsilon^\ell_n=\gamma_n\circ\Upsilon\circ\gamma_\ell.\]
The vector $\Upsilon^\ell_n(\rho)$ gives the lengths of the first $n$ external branches in the subtree spanned by the first $\ell$ leaves of the tree associated with some $\rho\in\Uu_N$ or $\rho\in\Uu$. We also define the restriction $\varpi_{\R_+^\N}:\hat\Uu\to\R_+^\N$, $(r,v)\mapsto v$. We endow $\R_+^n$ with the maximum norm and the induced metric. Let $\gamma_n:\R_+^\N\to\R^n_+$\label{Conv:not:gamma:R} be the restriction $v\mapsto(v(i))_{i\in\sn}$. Recall also the maps $\alpha$ and $\beta$ from Sections \ref{Conv:sec:dist-dec} and \ref{Conv:sec:mmm-nd}.
\begin{proof}[Proof of Theorem~\ref{Conv:thm:dust} (continuation)]
We set $\tilde\chi^N_0=\hat\chi^N_0$. Then, $\chi^N_0=\alpha(\tilde\chi^N_0)$ and we can assume that the processes $(\tilde\chi^N_k,k\in\N_0)$, $(\hat\chi^N_k,k\in\N_0)$, and $((r^N_k,v^N_k),k\in\N_0)$ are defined as in Remark~\ref{Conv:rem:coupl-mm-mod}. In particular, we have $\hat\chi^N_k=\beta\circ\alpha(\tilde\chi^N_k)$ for all $k\in\N_0$ a.\,s.

First we consider finite-dimensional distributions. Let $t\in\R_+$. For $N\geq\ell\geq n\geq 2$,
\begin{align*}
&\dP(\nu^{\hat\chi^N_{\fl{c_N^{-1}t}}},
\nu^{\tilde\chi^N_{\fl{c_N^{-1}t}}})\\
&\leq\dP(\gamma_n(\nu^{\beta\circ\alpha(\tilde\chi^N_{\fl{c_N^{-1}t}})}),
\gamma_n(\nu^{\tilde\chi^N_{\fl{c_N^{-1}t}}}))
+2^{-n+1}\\
&\leq\dP(\gamma_n(\nu^{N,\beta\circ\alpha(\tilde\chi^N_{\fl{c_N^{-1}t}})}),
\gamma_n(\nu^{N,\tilde\chi^N_{\fl{c_N^{-1}t}}}))
+2^{-n+1}+2n^2/N\\
&=\dP(\gamma_n\circ\beta\circ\alpha(\nu^{N,\tilde\chi^N_{\fl{c_N^{-1}t}}}),
\gamma_n(\nu^{N,\tilde\chi^N_{\fl{c_N^{-1}t}}}))
+2^{-n+1}+2n^2/N\\
&\leq 2\dP(\gamma_n\circ\Upsilon\circ\alpha(\nu^{N,\tilde\chi^N_{\fl{c_N^{-1}t}}}),
\gamma_n\circ\varpi_{\R_+^\N}(\nu^{N,\tilde\chi^N_{\fl{c_N^{-1}t}}}))
+2^{-n+1}+2n^2/N\quad\text{a.\,s.}
\end{align*}
For the first inequality, we use relation~\reff{Conv:eq:d-U-restr} and either \reff{Conv:eq:def-dP} or \reff{Conv:eq:dP-coupl}. For the last inequality, we use Lemma~\ref{Conv:lem:rv-v}, the definitions of $\Upsilon$ and $\beta$, and \reff{Conv:eq:def-dP} or \reff{Conv:eq:dP-coupl}. For the second inequality, we use the bounds
\[\dP(\gamma_n(\nu^{\beta\circ\alpha(\tilde\chi^N_{\fl{c_N^{-1}t}})}),
\gamma_n(\nu^{N,\beta\circ\alpha(\tilde\chi^N_{\fl{c_N^{-1}t}})}))\leq n^2/N\]
and
\[\dP(\gamma_n(\nu^{\tilde\chi^N_{\fl{c_N^{-1}t}}}),\gamma_n(\nu^{N,\tilde\chi^N_{\fl{c_N^{-1}t}}}))\leq n^2/N\]
which can be seen from the coupling characterization \reff{Conv:eq:dP-coupl} of the Prohorov metric. Here we couple sampling with and without replacement as in the proofs of e.\,g.\ \reff{Conv:eq:dm-Ndm} or of Lemma \ref{Conv:lem:cont}.

By definition of $\Upsilon^\ell_n$ and $\Upsilon$,
\[\gamma_n\circ\Upsilon(\rho)\leq\Upsilon^\ell_n(\rho)\]
for all $\rho\in\Uu_N$. Using Lemma~\ref{Conv:lem:ext-v}, we obtain
\begin{equation*}
|\gamma_n\circ\Upsilon\circ\alpha(r,v)-\gamma_n(v)|
\leq|\Upsilon^\ell_n\circ\alpha(r,v)-\gamma_n(v)|
\end{equation*}
for $\nu^{\tilde\chi^N_{\fl{c_N^{-1}t}}}$-a.\,a.\ $(r,v)\in\hat\Uu_N$ a.\,s. Again using the definition of the Prohorov metric, we obtain the first inequality in the following display.
\begin{align}
&\dP(\gamma_n\circ\Upsilon\circ\alpha(\nu^{N,\tilde\chi^N_{\fl{c_N^{-1}t}}}),
\gamma_n\circ\varpi_{\R_+^\N}(\nu^{N,\tilde\chi^N_{\fl{c_N^{-1}t}}}))\notag\\
&\leq\dP(\Upsilon^\ell_n\circ\alpha(\nu^{N,\tilde\chi^N_{\fl{c_N^{-1}t}}}),
\gamma_n\circ\varpi_{\R_+^\N}(\nu^{N,\tilde\chi^N_{\fl{c_N^{-1}t}}}))\notag\\
&\leq\dP(\Upsilon^\ell_n\circ\alpha(\nu^{\tilde\chi^N_{\fl{c_N^{-1}t}}}),
\gamma_n\circ\varpi_{\R_+^\N}(\nu^{\tilde\chi^N_{\fl{c_N^{-1}t}}}))
+\ell^2/N+n^2/N\quad\text{a.\,s.}
\label{Conv:eq:mmm-bound-dP}
\end{align}
For the second inequality, we again couple sampling with and without replacement, and we use the triangle inequality twice. By continuity of the maps $\chi\mapsto\nu^\chi$, $\alpha$, $\Upsilon^\ell_n$, $\varpi_{\R_+^\N}$, $\gamma_n$, and $\dP(\cdot,\cdot)$, and by Proposition~\ref{Conv:prop:mod},
the right-hand side of~\reff{Conv:eq:mmm-bound-dP} converges in distribution to
\begin{equation}
\label{Conv:eq:mmm-bound-dP-ln}
\dP(\Upsilon^\ell_n\circ\alpha(\nu^{\hat\chi_t}),
\gamma_n\circ\varpi_{\R_+^\N}(\nu^{\hat\chi_t}))
\end{equation}
as $N\to\infty$, where $(\hat\chi_s,s\in\R_+)$ is a $\hat\U$-valued $\Xi$-Fleming-Viot process with initial state $\hat\chi_0$. Here we also use that $(\hat\chi_s,s\in\R_+)$ has a.\,s.\ no discontinuity at the fixed time $t$.

For every $\rho\in\Uu$,
\[\lim_{\ell\to\infty}\Upsilon^\ell_n(\rho)=\gamma_n\circ\Upsilon(\rho)\]
by definition of $\Upsilon^\ell_n$, and by the definition of $\Upsilon:\Uu\to\R_+^\N$ in Section \ref{Conv:sec:dist-dec}. Hence, for every $\chi\in\hat\U$,
\[\lim_{\ell\to\infty}\dP(\Upsilon^\ell_n\circ\alpha(\nu^\chi),
\gamma_n\circ\Upsilon\circ\alpha(\nu^\chi))=0.\]
By \cite{Sampl}*{Proposition 3.3},
\[\Upsilon\circ\alpha(\nu^\chi)=\varpi_{\R_+^\N}(\nu^\chi).\]
Therefore, expression~\reff{Conv:eq:mmm-bound-dP-ln} converges to zero as $\ell$ tends to infinity. Now we let also $n$ tend to infinity. Then we can deduce that
\[\dP(\nu^{\hat\chi^N_{\fl{c_N^{-1}t}}},
\nu^{\tilde\chi^N_{\fl{c_N^{-1}t}}})\]
converges to zero in probability as $N$ tends to infinity.

As an immediate consequence, for $t_1,\ldots,t_k\in\R_+$, the vector
\[(\dP(\nu^{\hat\chi^N_{\fl{c_N^{-1}t_i}}},
\nu^{\tilde\chi^N_{\fl{c_N^{-1}t_i}}}),i\in\sk)\]
converges to zero in probability as $N$ tends to infinity. Using Proposition~\ref{Conv:prop:mod} again, Slutzky's theorem, and that $(\hat\chi_s,s\in\R_s)$ has a.\,s.\ no fixed times of discontinuity, we deduce the convergence in distribution
\[(\nu^{\hat\chi^N_{\fl{c_N^{-1}t_i}}},i\in\sk)\dto
(\nu^{\hat\chi_{t_i}},i\in\sk)\quad(N\to\infty).\]
Hence,
\[(\hat\chi^N_{\fl{c_N^{-1}t_i}},i\in\sk)\dto
(\hat\chi_{t_i},i\in\sk)\quad(N\to\infty)\]
in the marked Gromov-weak topology. As this convergence determines the finite-dimensional distributions of possible limit processes, it now suffices to show relative compactness of the sequence of processes $((\hat\chi^N_{\fl{c_N^{-1}t}},t\in\R_+),N\in\N)$, see \cite{EK86}*{Theorem 3.7.8}.
\end{proof}
To show the desired relative compactness in the proof of Theorem \ref{Conv:thm:dust}, we use the following lemma.
\begin{lem}
\label{Conv:lem:bound-MGP-rc}
Let $k\in\N_0$ and $N\geq 2$.
Then,
\[\dMGP(\hat\chi^N_k,\hat\chi^N_{k+1})\leq 2N^{-1}(N-\#\pi^N_{k+1})+c_N.\]
\end{lem}
The bound in Lemma \ref{Conv:lem:bound-MGP-rc} has the following meaning. There are at most $N-\#\pi^N_{k+1}$ many individuals in generation $k$ that have more than one offspring in generation $k+1$. For each such offspring, the associated external branch has length $c_N$ which needs not coincide with the external branch length of the ancestor in generation $k$. There are also $N-\#\pi^N_{k+1}$ individuals in generation $k$ that die, and each such death can drastically increase an external branch length in generation $k+1$ (this is the freeing phenomenon mentioned in Section \ref{Conv:sec:mod}). For the other individuals in generation $k$, the external branch lengths increase by $c_N$ from generation $k$ to $k+1$. As each individual has weight $N^{-1}$, the bound is a consequence of the definition of the marked Gromov-Prohorov metric and the coupling characterization of the Prohorov metric.
\begin{proof}
Let $L\subset\sN$ denote the set of the labels of the individuals of generation $k$ that have offspring in generation $k+1$, that is,
\[L=\{i\in\sN:\exists j\in\sN\text{ with }A_k(k+1,j)=i\}.\]
By definition of the population model in Section~\ref{Conv:sec:results},
\begin{equation}
\label{Conv:eq:K-pi}
\#L=\#\pi^N_{k+1}.
\end{equation}
For all $j_1,j_2\in\sN$ with $j_1\neq j_2$ and $i_1=A_k(k+1,j_1)$, $i_2=A_k(k+1,j_2)$, by definition of the population model in Section~\ref{Conv:sec:results},
\begin{equation}
\label{Conv:eq:bound-MGP-rhok}
\rho^N_{k+1}(j_1,j_2)=\rho^N_k(i_1,i_2)+2c_N.
\end{equation}
For $i\in\sN$, we define the set
\[C_i=\{j\in\sN\setminus\{i\}:\rho^N_k(i,j)=\min\{\rho^N_k(i,\ell):\ell\in\sN\setminus\{i\}\}\}.\]
In words, $C_i$ consists of the individuals other than $i$ with minimal distance to the individual $i$. That is, the set $C_i\cup\{i\}$ is the minimal clade of the individual $i$ in the sense of \cite{BF05}.
Moreover, we define
\[M=\{i\in\sN:C_i\cap L\neq\emptyset,\exists ! j\in\sN \text{ with }A_k(k+1,j)=i\}.\]
For $i\in M$, the individual $i$ of generation $k$ has exactly one offspring $j$ in generation $k+1$, and at least one other member of the minimal clade of $i$ has offspring in generation $k+1$. Hence, the minimal clade of $i$ in generation $k$ and the minimal clade of $j$ in generation $k+1$ have the same most recent common ancestor. This implies, for $i$ and $j$ as above,
\begin{equation}
\label{Conv:eq:F-clade}
\min_{\ell\in\sN\setminus\{j\}}\rho^N_{k+1}(j,\ell)
=\min_{\ell\in\sN\setminus\{i\}}\rho^N_{k}(i,\ell)+2c_N.
\end{equation}

We write $(r,v)=\beta(\rho^N_k)$ and $(r',v')=\beta(\rho^N_{k+1})$. For $i\in M$, let $d(i)$ denote the label of the unique descendant in generation $k+1$ of the individual $i$ of generation $k$. For all $i\in M$ and $j=d(i)$,
\begin{equation}
\label{Conv:eq:bound-freeing-v}
v'(j)=\tfrac{1}{2}\min_{\ell\in\sN\setminus\{j\}}\rho^N_{k+1}(j,\ell)
=\tfrac{1}{2}\min_{\ell\in\sN\setminus\{i\}}\rho^N_{k}(i,\ell)+c_N=v(i)+c_N
\end{equation}
by equation~\reff{Conv:eq:F-clade}.
For $i_1,i_2\in M$ with $i_1\neq i_2$ and $j_1=d(i_1)$, $j_2=d(i_2)$, it holds $j_1\neq j_2$, and by equations \reff{Conv:eq:bound-MGP-rhok} and \reff{Conv:eq:bound-freeing-v}
\begin{equation}
\label{Conv:eq:bound-freeing-r-rho-v}
r'(j_1,j_2)=\rho^N_{k+1}(j_1,j_2)-v'(j_1)-v'(j_2)
=\rho^N_{k}(i_1,i_2)-v(i_1)-v(i_2)=r(i_1,i_2).
\end{equation}

We define a relation $\fR$ between the semi-metric spaces $(\sN,r)$ and $(\sN,r')$ by
\[\fR=\{(i,d(i))\in\sN^2:i\in M\}.\]
Equation~\reff{Conv:eq:bound-freeing-r-rho-v} implies that the distortion of $\fR$ equals zero,
\[\dis\fR=\max\{|r(i_1,i_2)-r'(j_1,j_2)|:(i_1,j_1),(i_2,j_2)\in\fR\}=0.\]
We set
\[\hat\fR=\{((i,v(i)),(j,v'(j)))\in(\sN\times\R_+)^2:(i,j)\in\fR\}.\]
There exists a coupling $\nu$ of the probability measures $N^{-1}\sum_{i=1}^N\delta_{(i,v(i))}$ and $N^{-1}\sum_{j=1}^N\delta_{(j,v'(j))}$ on $\sN\times\R_+$ with
\[\nu(\hat\fR)\geq N^{-1}\#M=1-N^{-1}(N-\#M).\]
By equation~\reff{Conv:eq:bound-freeing-v}, it holds $|v(i)-v'(j)|\leq c_N$ for $(i,j)\in\fR$. By Proposition~\ref{Conv:prop:MGP-rel}, which also holds for marked semi-metric measure spaces, it follows
\[\dMGP(\hat\chi^N_k,\hat\chi^N_{k+1})\leq N^{-1}(N-\#M)+c_N.\]

It remains to show
\begin{equation}
\label{Conv:claim:bound-clade}
N-\#M\leq 2(N-\#L).
\end{equation}
The assertion then follows by equation~\reff{Conv:eq:K-pi}.

For $i\in\sN$, let $I_i=\I{i\in L,C_i\subset L^c}$. (Then $I_i$ is the indicator variable that individual $i$ reproduces as the only individual of its minimal clade.)
Let $i,j\in\sN$ with $i\neq j$ and consider the case that there exists $\ell\in C_i\cap C_j$. W.\,l.\,o.\,g., we assume $\rho^N_k(j,\ell)\leq\rho^N_k(i,\ell)$ (if this does not hold, we transpose $i$ and $j$). As $\rho^N_k\in\Uu$, we obtain
\[\rho^N_k(i,j)\leq\rho^N_k(i,\ell)\vee\rho^N_k(j,\ell)=\rho^N_k(i,\ell).\]
As $\ell\in C_i$, it follows $j\in C_i$. If $I_i=1$, then it follows that $j\in L^c$ and $I_j=0$. Hence, in any case, the elements of the set $\mathcal{A}:=\{C_i:i\in\sN,I_i=1\}$ are nonempty disjoint subsets of $L^c$, or it holds $\mathcal{A}=\emptyset$. This implies $\#\{C_i:i\in\sN,I_i=1\}\leq N-\#L$. Furthermore, generation $k$ contains at most $N-\#L$ many individuals with more than one offspring in generation $k+1$. The claim~\reff{Conv:claim:bound-clade} follows by definition of $M$.
\end{proof}

\begin{proof}[Proof of Theorem~\ref{Conv:thm:dust} (end)]
We assume $\Xi\in\dust$. To show relative compactness of the sequence of processes $((\hat\chi^N_{\fl{c_N^{-1}t}},t\in\R_+),N\in\N)$,
it suffices to verify condition (b) in Theorem 3.8.6 of \cite{EK86}. Condition (a) in this theorem is satisfied as the one-dimensional distributions converge.

Using Lemma~\ref{Conv:lem:bound-MGP-rc} and the
bound
\[N-\#\pi\leq\#\cup\sigma\]
for $\pi\in\p_N$ and $\sigma=\{B\in\pi:\# B\geq 2\}$, we obtain
\[\E[\dMGP(\hat\chi^N_k,\hat\chi^N_{k+1})]\leq 2N^{-1}\E[\#\cup\sigma^N_1]+c_N\]
for all $k\in\N_0$. By exchangeability,
\[\E[\#\cup\sigma^N_1]=\sum_{i=1}^N\E[\I{i\in\cup\sigma^N_1}]
=N\P(\gamma_1(\sigma^N_1)=\{\{1\}\}).\]
Let $(\F^N_t,t\in\R_+)$ be the filtration induced by the process $(\hat\chi^N_{\fl{c_N^{-1}t}},t\in\R_+)$. For $t\in\R_+$, $\delta>0$, $u\in[0,\delta]$, and $s\in[0,\delta\wedge t]$, the Markov property of $(\hat\chi^N_k,k\in\N_0)$ at $\fl{c_N^{-1}t}$ yields
\begin{align}
&\E[\dMGP(\hat\chi^N_{\fl{c_N^{-1}(t+u)}},\hat\chi^N_{\fl{c_N^{-1}t}})|\F^N_t]
\dMGP(\hat\chi^N_{\fl{c_N^{-1}(t-s)}},\hat\chi^N_{\fl{c_N^{-1}t}})\notag\\
&\leq \I{\delta\geq c_N/2}(\fl{c_N^{-1}\delta}+1)(2N^{-1}\E[\#\cup\sigma^N_1]+c_N)\notag\\
&\leq \I{\delta\geq c_N/2}(\delta+c_N)(2c_N^{-1}\P(\gamma_1(\sigma^N_1)=\{\{1\}\})+1)
\quad\text{a.\,s.}
\label{Conv:eq:rc-mmm-bound-Markov}
\end{align}
In the first inequality, we also use $\dMGP\leq 1$, and that if $\delta< c_N/2$, then at least one of the distances on the left-hand side of~\reff{Conv:eq:rc-mmm-bound-Markov} equals zero.

Now we show that the right-hand side of~\reff{Conv:eq:rc-mmm-bound-Markov} converges to zero uniformly in $N$ as $\delta$ tends to zero. For each $\ep>0$, there exists $N_\ep\geq 2$ such that for all $N\geq N_\ep$, it holds
\[c_N^{-1}\P(\gamma_1(\sigma^N_1)=\{\{1\}\})\leq 2\lambda_{1,\{\{1\}\}}\]
and $c_N<\ep$. Hence the right-hand side of~\reff{Conv:eq:rc-mmm-bound-Markov} is bounded from above by $(\delta+\ep)(4\lambda_{1,\{\{1\}\}}+1)$ for $N\geq N_\ep$. For $\delta$ sufficiently small and $N\leq N_\ep$, the right-hand side of~\reff{Conv:eq:rc-mmm-bound-Markov} equals zero.

As the right-hand side of~\reff{Conv:eq:rc-mmm-bound-Markov} does not depend on $t$, we have verified (8.28) and (8.29) in Theorem 3.8.6 of \cite{EK86}. To verify also (8.30), hence condition (b) in Theorem 3.8.6 of \cite{EK86}, we estimate as above
\[\E[\dMGP(\hat\chi^N_{\fl{c_N^{-1}\delta}},\hat\chi^N_0)]
\leq\fl{c_N^{-1}\delta}(2N^{-1}\E[\#\cup\sigma^N_1]+c_N)
\leq\delta(2c_N^{-1}\P(\gamma_1(\sigma^N_1)=\{\{1\}\})+1).\]
Also this expression converges to zero uniformly in $N$ as $\delta$ tends to zero.
\end{proof}

\section*{List of notation}
\sectionmark{List of notation}
Here we collect notation that is used globally in the article.\\

\small{
\hparagraph{Miscellaneous}
$\R_+=[0,\infty)$, $\N=\{1,2,3,\ldots\}$, $\N_0=\N\cup\{0\}$, $[N]=\{1,\ldots,n\}$ for $N\in\N$
\\
$\gamma_n$: restriction map in various contexts (p.\,\pageref{Conv:not:gamma:dm}, p.\,\pageref{Conv:not:gamma:part}, p.\,\pageref{Conv:not:gamma:R})\\
$\dP$: Prohorov metric\\

\hparagraph{(Marked) distance matrices}
$\Uu_N$, $\Uu$: space of semi-ultrametrics on $[N]$, on $\N$ (p.\,\pageref{Conv:not:Uu})\\
$\hat\Uu_N$, $\hat\Uu$: space of decomposed semi-ultrametrics on $[N]$, on $\N$ (p.\,\pageref{Conv:not:hatUuN}, p.\,\pageref{Conv:not:hatUu})\\
$\alpha$: retrieves the semi-ultrametric from a decomposed semi-ultrametric (p.\,\pageref{Conv:not:alphaN}, p.\,\pageref{Conv:not:alpha})\\
$\beta:\Uu_N\to\hat\Uu_N$: decomposition map into the external branches and the remaining subtree (p.\,\pageref{Conv:not:beta})\\
$\Upsilon(\rho)$: vector of the lengths of the external branches in the coalescent tree associated with $\rho$ (p.\,\pageref{Conv:not:UpsilonN}, p.\,\pageref{Conv:not:Upsilon})\\

\hparagraph{(Marked) metric measure spaces}
$\U_N$, $\U$: spaces of isomorphy classes of ultrametric measure spaces (p.\,\pageref{Conv:not:UN})\\
$\hat\U_N$, $\hat\U$: spaces of isomorphy classes of marked metric measure spaces (p.\,\pageref{Conv:not:hatUN})\\
$\dGP$, $\dMGP$: Gromov-Prohorov metric, marked Gromov-Prohorov metric (p.\,\pageref{Conv:not:dGP}, p.\,\pageref{Conv:not:dMGP})\\
$\nu^\chi$: distance matrix distribution of $\chi\in\U$, or marked distance matrix distribution of $\chi\in\hat\U$ (p.\,\pageref{Conv:not:nuchimm}, p.\,\pageref{Conv:not:nuchimmm})\\
$\nu^{N,\chi}$: $N$-distance matrix distribution of $\chi\in\U$, or $N$-marked distance matrix distribution of $\chi\in\hat\U$ (p.\,\pageref{Conv:not:nuNchimm}, p.\,\pageref{Conv:not:nuNchimmm})\\
$\UUerg$: space of distance matrix distributions (p.\,\pageref{Conv:not:UUerg})\\
$\psi_N:\Uu_N\to\U_N$, $\hat\psi_N:\hat\Uu_N\to\hat\U_N$: construction of (marked) metric measure spaces (p.\,\pageref{Conv:not:psi}, p.\,\pageref{Conv:not:hatpsi})\\
$\alpha:\hat\U_N\to\U_N$: maps a decomposed unlabeled tree to an unlabeled tree (p.\,\pageref{Conv:not:alphamm})\\
$\beta:\U_N\to\hat\U_N$: decomposes an unlabeled tree at the external branches (p.\,\pageref{Conv:not:betamm})\\
$\beta_0:\U_N\to\hat\U_N$: adds the zero mark (p.\,\pageref{Conv:not:beta0})\\
$\C_n$, $\hat\C_n$: sets of bounded differentiable functions with bounded uniformly continuous derivative (p.\,\pageref{Conv:not:Cn})\\
$\Pi$: set of polynomials on $\U$ (p.\,\pageref{Conv:not:Pi})\\
$\hat\Pi$: set of marked polynomials on $\hat\U$ (p.\,\pageref{Conv:not:hatPi})\\
$\Cc$: a set of test functions on $\UUerg$ (p.\,\pageref{Conv:not:Cc})\\

\hparagraph{Partitions and semi-partitions}
$\p_N$: Set of partitions of $[N]$, associated transformations (equation \reff{Conv:eq:pn-Un})\\
$\mathbf{0}_n=\{\{1\},\ldots,\{n\}\}\in\p_n$\\
$\#\pi$: number of blocks of a partition $\pi$\\
$\S_n$ set of semi-partitions of $[n]$ (p.\,\pageref{Conv:not:Sn}), associated transformations (p.\,\pageref{Conv:not:Sn-transf})\\
$\Delta=\{x=(x(1),x(2),\ldots):x(1)\geq x(2)\geq \ldots 0,|x|_1\leq 1\}$\\
$\Delta^N=\{x\in\Delta:|x|_1=1,Nx(i)\in\N_0\text{ for all }i\in\N\}$\\
$\Delta_c=\{x\in\Delta:x(1)>c\}$\\

\hparagraph{Genealogy in the Cannings model}
$(x^N_k,k\in\N)$: sequence in $\Delta^N$ that gives the family sizes (p.\,\pageref{Conv:not:xNk})\\
$(\pi^N_k,k\in\N)$ sequence in $\p_N$ that gives the families (p.\,\pageref{Conv:not:piNk})\\
$A_j(k,i)$: label of the ancestor in generation $j$ of the individual $i$ in generation $k$ (p.\,\pageref{Conv:not:anc})\\
$\rho^N_k(i,j)$: genealogical distance (p.\,\pageref{Conv:not:rhoNk})\\
$c_N$: pairwise coalescence probability (equation \reff{Conv:eq:cN})\\
$b_N$: probability that a randomly sampled individual is in a non-singleton family (p.\,\pageref{Conv:not:bN})\\
$\chi^N_k=\psi_N(\rho^N_k)$: unlabeled genealogical tree (equation \reff{Conv:eq:rho-chi})\\
$\hat\chi^N_k=\hat\psi_N(\beta(\rho^N_k))$: unlabeled genealogical trees, decomposed at the external branches (p.\,\pageref{Conv:not:hatchiNk})\\
$(r^N_k,v^N_k)$, $\tilde\chi^N_k=\hat\psi_N(r^N_k,v^N_k)$: another decomposition of the genealogical trees (p.\,\pageref{Conv:not:rvNk}, p.\,\pageref{Conv:not:tildechiNk})\\

\hparagraph{Tree-valued Fleming-Viot processes}
${\Mm_1(\Delta)}$, $\dust$, $\nd$: Set of probability measures on $\Delta$, subsets of the measures with and without dust (p.\,\pageref{Conv:not:MfD})\\
$\Xi=\Xi_0+\Xi\{0\}\delta_0$ (equation \reff{Conv:eq:Xi-dec})\\
$\lambda_{\pi}$, $\lambda_{n,\sigma}$: reproduction rates (p.\,\pageref{Conv:not:lambdapi}, p.\,\pageref{Conv:not:lambdansigma})\\
}

\normalsize
\noindent{ \textbf{Acknowledgments.}
This work is part of my PhD thesis. Partial support from the DFG Priority Programme 1590 ``Probabilistic Structures in Evolution'' is acknowledged.

\bibliography{C:/Users/sg/Documents/Bib/diss}
\bibliographystyle{plain}

\end{document}